\documentclass[10pt]{article}

\usepackage{amsfonts,graphicx,amsmath,amssymb,amsthm,geometry,bbm,hyperref,nicefrac,enumerate,comment,fontawesome,mathtools,cases,algorithmic,algorithm2e}

\usepackage{xcolor}

\DeclareMathOperator*{\argmin}{arg\,min}

\newcommand{\R}{\mathbb{R}}
\newcommand{\N}{\mathbb{N}}
\newcommand{\E}{\mathbb{E}}

\renewcommand{\P}{\mathbb{P}}

\newcommand{\F}{\mathcal{F}}

\newcommand{\e}{\text{e}}
\renewcommand{\d}{\text{d}}

\usepackage{array}
\newcolumntype{P}[1]{>{\centering\arraybackslash}p{#1}}
\newcommand{\triple}{{\vert\kern-0.25ex\vert\kern-0.25ex\vert}}

\usepackage{geometry}
\newtheorem{lemma}{Lemma}[section]

\newtheorem{assumption}[lemma]{Assumption}
\newtheorem{theorem}{Theorem}[section]
\newtheorem{corollary}{Corollary}[section]

\newtheorem{proposition}[lemma]{Proposition}

\theoremstyle{definition}

\newcommand{\footremember}[2]{%
    \footnote{#2}
    \newcounter{#1}
    \setcounter{#1}{\value{footnote}}%
}
\newcommand{\footrecall}[1]{%
    \footnotemark[\value{#1}]%
} 

\begin{document}

\title{Convergence of a robust deep FBSDE method for stochastic control}

\author{%
  Kristoffer Andersson\footremember{alley}{Research Group of Scientific Computing, Centrum Wiskunde \& Informatica} \and
  Adam Andersson\footremember{alleys}{Research Group of Computational Mathematics, Chalmers University of Technology \& the University of Gothenburg }\footremember{alleys2}{Saab AB Radar Solutions, Gothenburg, Sweden}
  \and Cornelis W. Oosterlee\footrecall{alley} \footremember{trailer}{Mathematical Institute, Utrecht University} 
 }

\maketitle
\begin{abstract}
In this paper, we propose a deep learning based numerical scheme for strongly coupled FBSDEs, stemming from stochastic control. It is a modification of the deep BSDE method in which the initial value to the backward equation is not a free parameter, and with a new loss function being the weighted sum of the cost of the control problem, and a variance term which coincides with the mean squared error in the terminal condition. We show by a numerical example that a direct extension of the classical deep BSDE method to FBSDEs, fails for a simple linear-quadratic control problem, and motivate why the new method works. Under regularity and boundedness assumptions on the exact controls of time continuous and time discrete control problems, we provide an error analysis for our method. We show empirically that the method converges for three different problems, one being the one that failed for a direct extension of the deep BSDE method.

\end{abstract}
\tableofcontents

\section{Introduction}
Forward backward stochastic differential equations (FBSDEs) constitute an important family of models with many applications in a wide variety of fields such as finance, physics, chemistry and engineering. As the name suggests, an FBSDE consists of two stochastic differential equations (SDE), one forward SDE and one backward SDE, commonly referred to as the forward equation and the backward equation, respectively. The forward equation is a classical SDE with a given initial value, while the backward equation has a given stochastic terminal value and the initial value is part of the solution. In this paper, we are concerned with stochastic control problems, which typically lead to coupled FBSDEs, meaning that the primary stochastic variables in the backward SDE impact the forward SDE and vice versa. Since closed form solutions to FBSDEs are rare, one often has to rely on numerical approximations. 

In this paper, we propose a method which falls into the rapidly growing category of neural network based approximation schemes for FBSDEs and PDEs. Even though there are a number of works in this direction, the methods stem from the pioneering work \cite{han2018solving}, and are similar in spirit. In the present paper, non-convergence of the method from \cite{han2018solving}, originally proposed for non-coupled FBSDEs, is identified, when applied to strongly coupled FBSDEs. We present a new family of methods and prove analytically and numerically that it overcomes the convergence problem.

Prior to the recent surge in machine learning based algorithms, the task of approximating FBSDEs has been an active field of research for several decades.
From a general perspective, approximation schemes can be categorized into \emph{backward} and \emph{forward} numerical methods, referring to the order in time in which the methods operate.

A backward numerical method usually relies on an initialization of the backward equation at the known terminal value (or an approximation of the terminal value). The solution is then approximated recursively, backwards in time, by approximating conditional expectations. There are several different methods to approximate these conditional expectations such as \textit{e.g.,} tree based methods, see \cite{teng2018review,ma2002numberical}, Fourier based methods, see \textit{e.g.,} \cite{ruijter2015fourier,huijskens2016efficient,ruijter2016numerical} and least-squares Monte Carlo (LSMC) methods, see \textit{e.g.,} \cite{bouchard2004discrete,chau2019stochastic,fahim2011probabilistic,bender2012least,crisan2010monte,gobet2016stratified}. A general property of backward methods is that the terminal condition of the backward equation depends on the realization of the forward equation. This is not a problem for \emph{decoupled} FBSDEs, but for coupled FBSDEs the method becomes implicit and iterative schemes which may not always converge need to be employed. A second class of algorithms, which is more suitable for coupled FBSDEs is the class of forward methods such as PDE based methods see \textit{e.g.,} \cite{ma1994solving,gobet2007error} and Picard linearization schemes, see \textit{e.g.,} \cite{gobet2010solving,hutzenthaler2019multilevel,bender2007forward}. For a summary on forward and backward numerical methods for FBSDEs, we refer to \cite{chessari2021numerical}. A typical drawback for the methods mentioned above is that they suffer from the curse of dimensionality (some methods such as LSMC and Picard schemes may overcome this problem to some extent), meaning that the time complexity and the memory requirements increase exponentially with the dimensions of the problem.

In addition to the classical methods described above, a new branch of algorithms based on neural networks has appeared in recent years. In \cite{han2018solving}, the authors present a method called the \emph{Deep BSDE method}, which relies on a neural network parametrization of the control process and the initial condition of the backward equation. Both the forward and the backward equations are then treated as forward equations and are approximated with the Euler--Maruyama scheme. To achieve an accurate approximation, the parameters of the neural networks are optimized such that the terminal condition on the backward equation is (approximately) satisfied in mean squared sense. The method has proven to be able to approximate a wide class of equations in very high dimensions (at least 100). Since the original Deep BSDE method publication, several papers with adjustments of the algorithm as in \textit{e.g.,} \cite{beck2021solving,beck2019machine,beck2021deep,fujii2019asymptotic,raissi2018forward,ji2020three,gnoatto2020deep,ji2021control, wang2022deep,henry2017deep}  and others with convergence analysis in \textit{e.g.,} \cite{han2020convergence,hutzenthaler2020proof, berner2020analysis,elbrachter2021dnn,Grohs2018APT,jentzen2018proof,jiang2021convergence}, have been published. In addition to being forward methods, these algorithms are \emph{global} in their approximation, meaning that the optimization of all involved neural networks is carried out simultaneously. This implies that they are optimized subject to one single objective function, also called \emph{loss function}, as it is usually referred to in the machine learning literature. There also exists a branch of neural network based algorithms relying on \emph{local} optimization techniques. Typically, these methods are of backward type and of similar nature as the LSMC algorithms, but instead of polynomials as their basis functions, neural networks are used. As for the LSMC method, these kinds of algorithms are not easily applied to coupled FBSDEs. Algorithms of this type can be found in \textit{e.g.,} \cite{hure2020deep,hure2019some,chan2019machine} and with error analysis \cite{fang2009novel,Balint}. For an overview of machine learning algorithms for approximation of PDEs, we refer to \cite{beck2020overview}.

As mentioned, we are interested in coupled FBSDEs and the focus is therefore on global algorithms operating forward in time, with a structure similar to the Deep BSDE method. We demonstrate that the approach taken in \textit{e.g.,} \cite{andersson2019approximate,pereira2019learning,liu2021deep,ji2020three,dai2021learning}, where the deep BSDE method is applied to the FBSDEs associated with the stochastic control problem, is problematic. As we show in the present paper, even though an accurate approximation of the control problem can be achieved, this does not imply an approximation of the FBSDE in general. Moreover, it is observed that the deep BSDE method does not converge for certain problems and the convergence problem is prominent for the strongly coupled FBSDEs. Our proposed method overcomes this problem by employing the equivalence between the stochastic control problem and the FBSDE. To be more precise, we use the fact that the initial value of the BSDE coincides with the value function of the control problem and hence, can be expressed as a minimization problem. Moreover, we use the adaptivity property of the BSDE to conclude that a stochastic version of the value function is $\F_0-$measurable and therefore has zero variance. These two properties are then combined in the loss function to achieve a robust approximation scheme for coupled FBSDEs. The effectiveness of our algorithm is demonstrated  empirically on a collection of problems with different characteristics. In addition, a theoretical error analysis is carried out, in which we provide convergence rates for the initial and terminal conditions of the FBSDE under a mild assumption and strong convergence of the FBSDE under stronger assumptions. Our main result is similar to the aposteriori error bound for the deep BSDE method which was established for weakly coupled FBSDEs in \cite{han2020convergence} and later extended to non-Lipschitz coefficients (but for less general diffusion coefficients) in \cite{jiang2021convergence}. However, these results are unlikely to be valid for strongly coupled FBSDEs and we find several examples in which the discrete terminal condition converges while the FBSDE approximation does not.

This paper is structured as follows: In Section~\ref{sec2}, we present the stochastic control problem and explain the reformulations to a PDE as well as a FBSDE. Moreover, we introduce reformulations of the FBSDE to different variational problems which are used for the algorithms in later sections. The section concludes by introducing the time discrete counterparts of the reformulations as well as a discussion on when and why the deep BSDE method fails to converge. In Section~\ref{sec3}, the fully implementable algorithms are presented together with details on the neural networks used. Section~\ref{sec4} is devoted to error analysis of the proposed algorithm. 
Classical Euler--Maruyama type discretization errors and errors stemming from differences between time discrete and time continuous stochastic control are discussed. Finally, in Section~\ref{sec5} numerical approximations are compared with their analytic counterparts when we have such available. 

\section{The deep FBSDE method and an improved family of methods}\label{sec2}

This section contains a formal introduction to our proposed method with motivation from stochastic control and FBSDE theory. Section~\ref{sec:FBSDE} introduces the stochastic control problem, the related Hamilton-Jacobi-Bellman equation and FBSDE. Alternative formulations of the FBSDE are presented in Section~\ref{sec:alt_form}. In Section~\ref{sec:deepBSDE}, we motivate by numerical examples why a direct extension of the deep BSDE method to FBSDEs, as in \cite{andersson2019approximate,pereira2019learning,liu2021deep}, fails for many problems. Finally, in Section~\ref{sec:deepFBSDE}, the proposed robust deep FBSDE method is described. In this section, we present formally the method for the sake of clarity while in Section \ref{sec4}, a more rigorous approach is taken.

\subsection{Stochastic control and FBSDEs}\label{sec:FBSDE}

Our starting point is a controlled SDE and its associated cost functional
\begin{numcases}{}\label{SDEu}
    X_t^u = x_0 + \int_0^t \bar{b}(s,X_s^u,u_s)\d s + \int_0^t\sigma(s,X_s^u)\d W_s,\\
\label{J}
        J^u(t,x)=\E\bigg[\int_t^T \bar{f}(s,X_s^u,u_s)\d s + g(X_T^u)\,\Big|\,X_t^u=x\bigg],\quad t\in[0,T].
\end{numcases}Here $T\in(0,\infty)$, $d,k,\ell\in\N$, $(W_t)_{t\in[0,T]}$ is a $k$-dimensional standard Brownian motion, the coefficients $\bar{b}\colon[0,T]\times \R^d\times\R^\ell\to\R^d$, $\sigma\colon[0,T]\times \R^d\times\R^\ell\to\R^{d\times k}$, $\bar{f}\colon[0,T]\times \R^d\times\R^\ell\to\R$ and $g\colon\R^d\to\R$ satisfy some extra regularity conditions, and the control process, $u=(u_t)_{t\in[0,T]}$, belongs to a set $\mathcal U$ of admissible controls, taking values in a set $U\subset \R^\ell$. The aim is to find a control process, $u^*\in\mathcal U$, that minimizes $J^u(0,x_0)$. 

Assuming the cost to be bounded from below, the \emph{value function} of the control problem is given by \begin{equation}
    \label{value_fun}
    V(t,x)=\inf_{u\in\mathcal{U}}J^u(t,x).
\end{equation}
For the presentation, we assume uniqueness of the infimum. Under appropriate conditions, the value function satisfies a \emph{Hamilton--Jacobi--Bellman} (HJB) equation, which is a non-linear parabolic PDE given by \begin{equation}\label{HJB}\begin{dcases}
    \frac{\partial V}{\partial t}(t,x) + \frac{1}{2}\text{Tr}(\sigma\sigma^\top\mathrm{D}_x^2V)(t,x)+ \mathcal{H}(t,x,\text{D}_xV(t,x))=0, & (t,x)\in [0,T)\times \R^d,\\
    V(t,x)=g(x), & (t,x)\in\{T\}\times \R^d.
    \end{dcases}
\end{equation}
Here $\mathrm{Tr}$ denotes the trace of a matrix and for $(t,x)\in [0,T]\times \R^d$, $p\in\R^d$ the \emph{Hamiltonian}, $\mathcal{H}$, is given by \begin{equation}\label{Hamiltonian}\mathcal{H}(t,x,p)=\inf_{v\in U}\big[\bar{b}(t,x,v)^\top p+\bar{f}(t,x,v)\big].\end{equation} 
Under conditions that guarantee a sufficiently regular solution to \eqref{HJB}, and the infimum in the Hamiltonian to be attained at $v^*=v^*(t,x,p)$, the optimal control is of the feedback form  $u^*(t,X_t)=v^*\big(t,X_t,\text{D}_xV(t,X_t)\big)$, where we have written $X\coloneqq X^{u^*}$ for the optimally controlled $X^u$. The feedback map $v^*$ is for many interesting problems easy to derive. Again, under sufficient regularity, It\^o's formula applied to $V(t,X_t)$ yields
\begin{equation}\label{FBSDE}\begin{dcases}
X_t=x_0+\int_0^t b(s,X_s,Z_s)\d s + \int_0^t\sigma(s,X_s)\d W_s,\\
Y_t=g(X_T) + \int_t^T f(s,X_s,Z_s)\d s -\int_t^T \langle Z_s, \d W_s\rangle,\quad t\in[0,T],
    \end{dcases}
\end{equation}
where $Y_t=V(t,X_t)$, $Z_t=\sigma^T(t,X_t)\mathrm{D}_xV(t,X_t)$ and, for $\theta\in\{b,f\}$, we have
\begin{align*}
    \theta(t,X_t,Z_t)\coloneqq \bar{\theta}(t,X_t,v^*(t,X_t, (\sigma(t,X_t)\sigma^T(t,X_t))^{-1}\sigma(t,X_t)Z_t)).
\end{align*}
In the rest of this section, we assume the existence of a unique solution $(X,Y,Z)$ of \eqref{FBSDE} in appropriate spaces. Given $Z$, or equivalently $\mathrm{D}_xV$, we thus have an optimal control $u_t^*=v^*(t,X_t, (\sigma(t,X_t)\sigma^T(t,X_t))^{-1}\sigma(t,X_t)Z_t)$. This would make efficient numerical FBSDEs schemes very useful for solving the control problem. In the other direction, if we have an optimal control $u^*$, then in general this does not give us $Z$, unless $p\mapsto v^*(t,x,p)$ is invertible, and only in this case the control problem naturally suggests numerical schemes for FBSDEs. Below, we introduce a family of numerical schemes for FBSDEs that works regardless of invertibility of the feedback map, but reduces to the control problem in the case of invertibility.

\subsection{Alternative formulations of FBSDEs}\label{sec:alt_form}
The Deep BSDE method proposed in \cite{han2018solving}, relies on a reformulation of the FBSDE \eqref{FBSDE} into two forward SDEs, one with apriori unknown initial value.  It relies moreover on the Markov property of the FBSDE, which guarantees that $Z_t=\zeta^*(t,X_t)$, for some function $\zeta^*\colon[0,T]\times\R^d\to\R^k$, that we refer to as Markov map, and optimization is done with respect to such functions and initial values $y_0$. More precisely, the FBSDE \eqref{FBSDE} is reformulated into the following variational problem
\begin{equation}\label{var0_FBSDE}\begin{dcases}
\underset{y_0,\zeta}{\mathrm{minimize}}\ 
 \E|Y_T^{y_0,\zeta}-g(X_T^{y_0,\zeta})|^2,\quad \text{where}\\
X_t^{y_0,\zeta}=x_0+\int_0^tb(s,X_s^{y_0,\zeta},Z^{y_0,\zeta}_s)\d s + \int_0^t\sigma(s,X_s^{y_0,\zeta})\d W_s,\\
Y_t^{y_0,\zeta}= y_0-\int_0^tf(s,X_s^{y_0,\zeta},Z^{y_0,\zeta}_s)\d s +\int_0^t \langle Z^{y_0,\zeta}_s,\d W_s\rangle,\\
Z^{y_0,\zeta}_t=\zeta(t,X_t^{y_0,\zeta}),\quad t\in[0,T],
    \end{dcases}
\end{equation}
where $y_0$ and $\zeta$ are sought in appropriate spaces. From the theory outlined in Section~\ref{sec:FBSDE}, under well-posedness and sufficient regularity of \eqref{FBSDE}, it holds that $Y_0=V(0,x_0)$ and $\zeta^*=\sigma^\top\mathrm{D}_x V$, and thus we have well-posedness of \eqref{var0_FBSDE}. While it seems natural to propose a numerical algorithm based on a discrete version of \eqref{var0_FBSDE}, we demonstrate below that such an optimization problem, even for many simple problems, does not converge.  

In order to introduce numerical schemes that do not suffer under the above problem, we use the following two properties of the initial value $Y_0$ of \eqref{FBSDE}:
\begin{enumerate}[(i)]
    \item $Y_0$ coincides with the value function of the control problem (property from the control problem);
    \item $Y_0$ is $\F_0-$measurable and therefore has zero variance (property from the FBSDE).
\end{enumerate}
The two properties are both captured in the following variational problem:
\begin{equation}\label{var_FBSDE}\begin{dcases}
\underset{\zeta}{\mathrm{minimize}}\ 
\Phi_\lambda(\zeta)=\E[\mathcal{Y}_0^\zeta] + \lambda\text{Var}[\mathcal{Y}_0^\zeta],\quad\mathrm{where}\\
\mathcal{Y}_0^\zeta=g(X_T^\zeta)+\int_0^Tf(t,X_t^\zeta,Z_t^\zeta)\d t-\int_0^T\langle Z^\zeta_t,\d W_t\rangle,\\
X_t^\zeta=x_0+\int_0^tb(s,X_s^\zeta,Z^\zeta_s)\d s + \int_0^t\sigma(s,X_s^\zeta)\d W_s,\\
Y_t^\zeta= \E[\mathcal{Y}_0^\zeta]-\int_0^tf(s,X_s^\zeta,Z^\zeta_s)\d s +\int_0^t \langle Z^\zeta_s,\d W_s\rangle,\\
Z_t^\zeta = \zeta(t,X_t^\zeta),\quad t\in[0,T].
    \end{dcases}
\end{equation}
We refer to $\mathcal{Y}_0^\zeta$ as the \emph{stochastic cost} and notice that $\E[\mathcal{Y}_0^\zeta]=J^{u(\zeta)}(0,x_0)$, where $u(\zeta)\in\mathcal U$ is the control generated by $\zeta$. Thus, the first term of the objective function $\Phi_\lambda$ is the cost function of the control problem. In the case of $p\mapsto v^*(t,x,p)$ being invertible, this term alone, i.e., for $\lambda=0$, offers an equivalent formulation to \eqref{var0_FBSDE}. In other cases, uniqueness of minimizers of $\zeta\mapsto \E[\mathcal{Y}_0^\zeta]$ cannot be guaranteed, but among the minimizers, there is only one $\zeta^*$ with the property that the variance of the stochastic cost $\mathcal{Y}_0^\zeta$ equals zero. The second term of $\Phi_\lambda$ is introduced to penalize non-zero variance and the minimizer for $\lambda>0$ is unique. Another important feature of the formulation \eqref{var_FBSDE}, is that $Y_0^\zeta$ is determined by $\zeta$ alone and \eqref{var_FBSDE} has thus one degree of freedom less than \eqref{var0_FBSDE}. A final observation is that
\begin{align}\label{eq:Var_identity}
    \text{Var}[\mathcal{Y}_0^\zeta]
    =\E\big[\big|\E[\mathcal{Y}_0^\zeta]-\mathcal{Y}_0^\zeta\big|^2\big]
    &=\E\Bigg[\bigg|Y_0^\zeta-\int_0^tf(s,X_s^\zeta,Z_s^\zeta)\d s +\int_0^t Z_s^\zeta\d W_s-g(X_T^\zeta)\bigg|^2\Bigg]    
    =
    \E\big[|Y_T^\zeta-g(X_T^\zeta)|^2\big].
\end{align}
This implies that the second term of $\Phi$ is, up to the factor $\lambda$, the same as that of \eqref{var0_FBSDE}, but with $Y_0^\zeta$ not being a variable to optimize. Thus, there are strong similarities between \eqref{var0_FBSDE} and \eqref{var_FBSDE}, but in the time discrete setting the latter formulation is shown to be advantageous sections below.

\subsection{A direct extension of the deep BSDE method and why it fails}\label{sec:deepBSDE}
In this section, we present the time discrete counterpart of \eqref{var0_FBSDE}. We assume an equidistant time grid, $0=t_0<t_1<\ldots<t_N=T$, with $h=t_{n+1}-t_n$ and denote the Brownian increment $\Delta W_n=W_{n+1}-W_n$. Throughout the paper, we parameterize discretizations by $h\in(0,1)$ and by this we mean all $h\in(0,1)\cap\{T/N:N\geq1\}$.

The time discrete version of \eqref{var0_FBSDE} is given by
\begin{equation}\label{var0_discrete_FBSDE}\begin{dcases}
\underset{y_0,\zeta}{\mathrm{minimize }}\;\;
 \E
 \Big[
   \big|
     Y_N^{h,y_0,\zeta}-g(X_N^{h,y_0,\zeta})
   \big|^2
 \Big],\quad \text{where}\\
X_n^{h,y_0,\zeta}=x_0+\sum_{k=0}^{n-1}b\big(t_k,X_k^{h,y_0,\zeta},Z_k^{h,y_0,\zeta}\big)h + \sum_{k=0}^{n-1}\sigma(t_k,X_k^{h,y_0,\zeta})\Delta W_k,\\
Y_n^{h,y_0,\zeta}= y_0-\sum_{k=0}^{n-1}f\big(t_k,X_k^{h,y_0,\zeta},Z_k^{h,y_0,\zeta}\big)h +\sum_{k=0}^{n-1} \big\langle Z_k^{h,y_0,\zeta},\Delta W_k\big\rangle,\\
Z_k^{h,y_0,\zeta}=\zeta_k(X_k^{h,y_0,\zeta}).
    \end{dcases}
\end{equation}

It is a direct extension of the deep BSDE method from \cite{han2018solving}. In the literature, it was first applied experimentally to FBSDEs in the master thesis \cite{andersson2019approximate} and thereafter in \cite{pereira2019learning}, both for inverted pendulums, in \cite{liu2021deep} for an application to attitude control of unmanned aerial vehicles. More examples of implementations of the deep FBSDE method are found in \cite{ji2020three,dai2021learning}.

In \cite{han2020convergence}, the authors consider FBSDEs with coefficients $b,\sigma$ and $f$ that may take the $Y$-component, but not the $Z$-component, as arguments. Under the relatively strict assumption of weak coupling (also called monotonicity condition, see \textit{e.g.,} \cite{antonelli1993backward}), it is shown that for $h$ small enough there is a constant $C$, independent of $h$, such that
\begin{equation}\label{eq:HanLong}
    \sup_{t\in[0,T]}
    \Big(
      \E\Big[
        \big|
          X_t-\hat{X}_t^{h,y_0,\zeta}
        \big|^2
      \Big]
        + 
        \E
        \Big[
          \big|
            Y_t-\hat{Y}_t^{h,y_0,\zeta}
          \big|^2
        \Big]
    \Big) 
    + 
    \int_0^T
    \E
    \Big[
      \big|
        Z_t-\hat{Z}_t^{h,y_0,\zeta}
      \big|^2
    \Big]\d t
    \leq C
    \Big(h + 
    \E
    \Big[\big|
      Y_N^{h,y_0,\zeta}-g(X_N^{h,y_0,\zeta})
    \big|^2\Big]\Big),
\end{equation}
where for $t\in[t_k,t_{k+1})$, $\hat{X}_t^{h,y_0,\zeta}\coloneqq X_k^{h,y_0,\zeta}$, $\hat{Y}_t^{h,y_0,\zeta}\coloneqq Y_k^{h,y_0,\zeta}$ and $\hat{Z}_t^{h,y_0,\zeta}=\zeta_k(X_k^{h,y_0,\zeta})$. Under some additional assumptions on the coefficients $b,f,\sigma$ and $g$ (additional smoothness and boundedness of the coefficients to guarantee a bounded and smooth solution of the associated HJB equation), the results in \cite{han2020convergence} can be extended to the framework of interest in this paper, \textit{i.e.,} coefficients taking the $Z$-component as an argument. On the other hand, the weak coupling condition is rarely satisfied for FBSDEs stemming from stochastic control problems, and to the best of our knowledge, there is no known way to relax this condition.  

To investigate convergence of \eqref{var0_discrete_FBSDE} empirically, we first note that if we would know $Y_0$  apriori, then the variational problem \eqref{var0_FBSDE} would be reduced to finding $\zeta$. We also know that $\mathrm{D}_xV^\top\sigma$ is the minimizer, which would make the objective function identical to zero. In the discrete counterpart, we would expect that, if \eqref{var0_discrete_FBSDE} converges to \eqref{var0_FBSDE}, then for sufficiently small $h$, the objective function would be close to zero if optimizing only $\zeta$ and setting $y_0=Y_0$. Moreover, for a robust algorithm to emerge from \eqref{var0_discrete_FBSDE}, it is important that $y_0\neq Y_0$ results in a larger value of the optimal objective function, at least when $y_0$ and the true initial value, $Y_0$, are ''far away'' from each other. To formalize this, we introduce the mean squared error
\begin{equation}\label{var_empirical_investigation}
  \mathrm{MSE}(y_0)
  \coloneqq
  \underset{\zeta}{\inf}\ 
  \E
  \big[
    \big|
      Y_N^{h,y_0,\zeta}-g(X_N^{h,y_0,\zeta})
    \big|^2
  \Big].
\end{equation}
The aim is to investigate whether or not $\text{MSE}$ is minimized at, or close to, the true initial condition $Y_0$. Moreover, for each $y_0$, we want to investigate the Markov map $\zeta^{y_0}$ that minimizes 
$ \zeta\mapsto 
  \E[|
      Y_N^{h,y_0,\zeta}-g(X_N^{h,y_0,\zeta})
    |^2]
$. The discrete costs, associated with $(y_0,\zeta^{y_0})$ and $(y_0,\zeta)$, are given by
\begin{equation*}\label{cost_disc}
    J^{h,y_0}_0
    =
    J^{h,y_0,\zeta^{y_0}}_0
    \quad\textrm{and}\quad
    J^{h,y_0,\zeta}_0
    =
    \E
    \big[
      \mathcal{Y}^{h,y_0,\zeta}_0
    \big].
\end{equation*}
Here, the \emph{discrete stochastic cost} is given by
\begin{align}\label{eq:stoch_cost}
  \mathcal{Y}^{h,y_0,\zeta}_0
  =
  g\big(X_N^{h,y_0,\zeta}\big)+\sum_{k=0}^{N-1}f\big(t_k,X_k^{h,y_0,\zeta},Z_k^{h,y_0,\zeta}\big)h -\sum_{k=0}^{N-1} \big\langle Z_k^{h,y_0,\zeta},\Delta W_k\big\rangle.
\end{align}
Using the stochastic cost, we have by a substitution that
\begin{equation}\label{var_empirical_investigation2}
  \mathrm{MSE}(y_0)
  =
  \underset{\zeta}{\inf}\ 
  \E
  \big[
    \big|
    \mathcal{Y}_0^{h,y_0,\zeta}
    - y_0
    \big|^2
  \big].
\end{equation}

Figure~\ref{MSE} shows $y_0\mapsto\mathrm{MSE}(y_0)$ and $y_0\mapsto J_0^{h,y_0}$ for two different Linear-Quadratic (LQ) control problems respectively, a one-dimensional and a two-dimensional problem. The left figure shows that there is a minimum of $\mathrm{MSE}$ at the correct $Y_0$ and for this problem the method converges. For the right figure, it is clear that there is no minimum of $\mathrm{MSE}$ around $Y_0$, or anywhere in the range. When $y_0$ and $\zeta$ are jointly optimized, the method has no chance to converge for this problem. 

Both problems considered in this section are of the form \eqref{FBSDE_LQG} with parameters as in Section~\ref{del} for the two-dimensional problem and $A=B=C=R_x=R_u=G=1$ and $\sigma=0.5$ for the one-dimensional problem. 

\begin{figure}[htp]
\centering
\begin{tabular}{cc}
          \includegraphics[width=80mm]{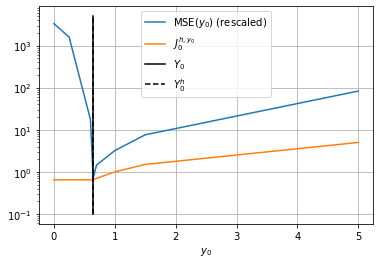}&    
          \includegraphics[width=80mm]{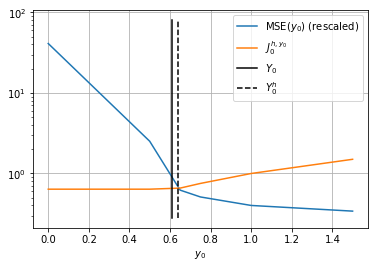}
\end{tabular}
\caption{Demonstration of the performance of the direct extension of the deep BSDE method to FBSDEs corresponding to two LQ control problems. \textbf{Left:} A one-dimensional problem with $N=10$ time steps. \textbf{Right:} A two-dimensional problem with $N=100$ time steps.}\label{MSE}
\end{figure}

We identify three distinct cases from Figure~\ref{MSE}:
\begin{itemize}
    \item 
    $y_0\approx Y_0$: In this case, $\mathrm{MSE}$ is the mean squared error of the discretized FBSDE and it is most reasonable that an approximation of $(X,Y,Z)$ is obtained by optimizing over $\zeta$.
    \item
    $y_0>Y_0$: Under this assumption, any $\zeta$ attaining the infimum in \eqref{var_empirical_investigation} satisfies $J^{h,y_0}_0=y_0$. Thus minimizing \eqref{var_empirical_investigation} is the same as finding the $\zeta$ that generates discrete cost $y_0$ and that at the same time minimizes the mean squared error in the terminal condition. A $\zeta$ with cost $y_0>Y_0$ has the possibility to generate a lower $\mathrm{MSE}(y_0)<\mathrm{MSE}(Y_0)$, depending on $y_0$ and the problem at hand. This is only possible if the coupling of $Z$ in $b$ is strong enough, so that $g(X_T)$ can be efficiently controlled by $Z$. In the case with no coupling, i.e., for a BSDE, $y_0>Y_0$ leads to an $\mathrm{MSE}$ of  the magnitude $y_0-Y_0$, and thus $\mathrm{MSE}$ is increasing in this regime. This is the reason why non-coupled BSDEs, as in \cite{han2018solving}, or weakly coupled FBSDEs, as in \cite{han2020convergence}, can be approximated with the deep BSDE method. The left plot of Figure~\ref{MSE} shows this favorable behaviour, while the right plot has a decreasing $\mathrm{MSE}$ and has no chance to converge. We also see that the cost increases linearly for $y_0>Y_0$ according to $J^{h,y_0}_0=y_0$. 
    \item
    $y_0<Y_0$: Since $Y_0$ is (approximately) a lower bound of $(y_0,\zeta)\mapsto J^{h,y_0,\zeta}$, it holds that any $\zeta$ attaining the infimum in $\zeta\mapsto  \E
  [|Y_N^{h,y_0,\zeta}-g(X_N^{h,y_0,\zeta})|^2]$ also minimizes the cost functional $\zeta\mapsto J^{h,y_0,\zeta}_0$. But $y_0$ does not enter $\mathcal{Y}^{h,y_0,\zeta}_0$ explicitly. Therefore, the minimizer of $\zeta\mapsto\E[\mathcal{Y}^{h,y_0,\zeta}_0]$ does not depend on $y_0$. Thus, fixing $y_0<Y_0$ and optimizing $\zeta$ approximates a solution to the control problem but not to the FBSDE. This can clearly be seen in Figure~\ref{MSE} from the cost being constant for $y_0<Y_0$ in both plots. It is also clear that the $\mathrm{MSE}$ increases for decreasing $y_0<Y_0$.
\end{itemize}
To further visualize the three cases, Figure~\ref{Fy0} shows the empirical means and
90\% credible regions (defined as the area between the $5:$th and the $95:$th empirical percentiles at each time point) for the true and approximated $Y$ and $Z$ processes of the two-dimensional LQ control problem discussed above. In the top row, we see that, in the case $y_0\approx Y_0$, the two components of the $Z$-process are very well approximated, but the time discretization error of $Y$ is visible. In the middle row, for $y_0>Y_0$, we see what is expected based on the discussion above. The $Y$ process satisfies the terminal value but is otherwise fundamentally distinct from the true $Y$, and $Z$ is different and oscillating (it is specified to have cost $y_0$). In the bottom row, the case $y_0<Y_0$ is shown. Just as explained above, the control problem is solved and therefore the $Z$ process is well approximated. It should though be noted that this is only true since the map $p\mapsto v^*(t,x,p)$, for this specific problem, is invertible. Otherwise, one optimal $\zeta$, in a set of many optimal Markov maps, is approximated. Thus, the control problem is approximately solved, however, the $Z$-component of the FBSDE is unlikely to be accurate. The $Y$ process is shifted by $y_0-Y_0$ and the terminal value is naturally not satisfied.

\begin{figure}[htp]
\centering
\begin{tabular}{ccc}
\includegraphics[width=50mm]{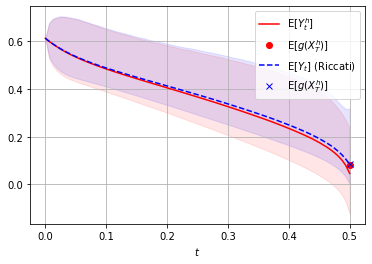}&
          \includegraphics[width=50mm]{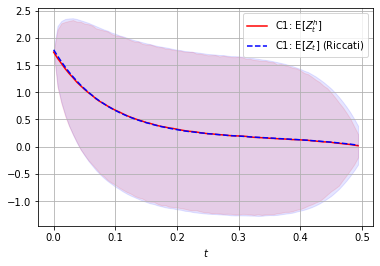}&
          \includegraphics[width=50mm]{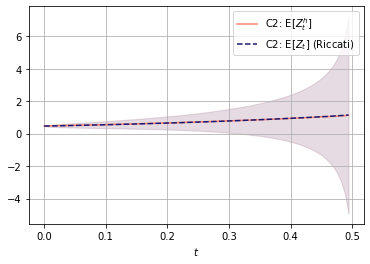}\\
\includegraphics[width=50mm]{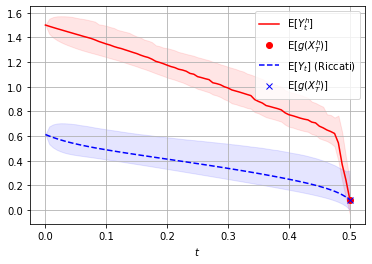}&
          \includegraphics[width=50mm]{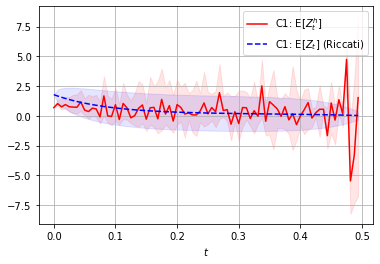}&
          \includegraphics[width=50mm]{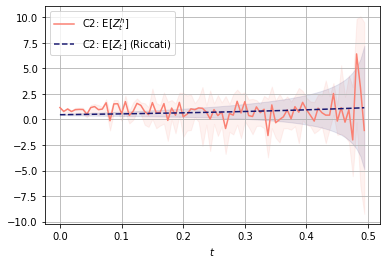}\\
\includegraphics[width=50mm]{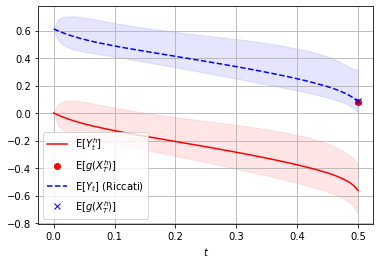}&
          \includegraphics[width=50mm]{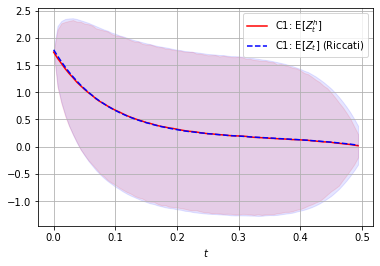}&
          \includegraphics[width=50mm]{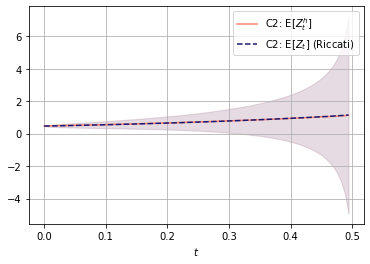}
\end{tabular}
\caption{Demonstration of the performance of the Deep FBSDE solver with fixed initial condition for $N=100$ time steps. The shaded areas represent the domain in which $90\%$ of all trajectories lie (the area is bounded by the 5:th and the 95:th empirical percentiles). \textbf{Left} to \textbf{right}: Sample means of the approximate and the semi-analytic (Riccati solutions) $Y-$process and the first and and second components of the $Z-$process (C1 and C2). \textbf{Top} to \textbf{bottom}: Initial conditions $\hat{y}_0=0.612\approx Y_0$,  $\hat{y}_0=1.5>Y_0$ and $\hat{y}_0=0.0<Y_0$.}
\label{Fy0}
\end{figure}
When searching for interesting example problems, we learned that it is much easier to find problems that do not converge than finding problems that do converge.

\subsection{A robust deep FBSDE method}\label{sec:deepFBSDE}
Having observed the problems with the direct extension of the deep BSDE method to FBSDEs, we here discretize the alternative formulation \eqref{var_FBSDE} of the FBSDE to obtain an alternative family of deep FBSDE methods. It reads:
\begin{equation}\label{var_discrete_FBSDE}\begin{dcases}
\underset{\zeta}{\mathrm{minimize}}\ 
 \Phi_{\lambda,h}(\zeta)
 =
 \E[\mathcal{Y}_0^{h,\zeta}]+\lambda\text{Var}[\mathcal{Y}_0^{h,\zeta}],\quad \text{where},\\
 \mathcal{Y}_0^{h,\zeta}:=g(X_N^{h,\zeta})+\sum_{k=0}^{N-1}f(t_k,X_k^{h,\zeta},Z_k^{h,\zeta})h 
-\sum_{k=0}^{N-1}\big\langle Z_k^{h,\zeta},\Delta W_k\big\rangle,\\
X_n^{h,\zeta}=x_0+\sum_{k=0}^{n-1}b(t_k,X_k^{h,\zeta},Z_k^{h,\zeta})h + \sum_{k=0}^{n-1}\sigma(t_k,X_k^{h,\zeta})\Delta W_k,\\
Y^{h,\zeta}_n=\E[\mathcal{Y}^{h,\zeta}_0]-\sum_{k=0}^{n-1}f(t_k,X^{h,\zeta}_k,Z_k^{h,\zeta})h +\sum_{k=0}^{n-1} \big\langle Z_k^{h,\zeta},\Delta W_k\big\rangle,\\ Z_k^{h,\zeta}=\zeta_k(X_k^{h,\zeta}).
    \end{dcases}
\end{equation}
The main purpose of the current paper is the theoretical and numerical error analysis of \eqref{var_discrete_FBSDE}. 

\subsection{Related methods and comparison} 
In a comparison with the current literature we focus on methods for solving stochastic control problem, or the associate FBSDE, by deep learning in a global way, in the sense that only one global optimization problem is solved. 

In the early paper \cite{han2016deep}, time discrete stochastic optimal control problems were solved with deep learning. No explicit connections to FBSDEs, or even to stochastic control in continuous time, were made. The feedback maps for the controls at all time steps were optimized over a family of neural networks, to minimize the discrete cost functional. This methodology is similar to our method when $\lambda=0$. It only differs in its approximation of the feedback map for $u$ instead of the Markov map for $Z$. The connection between our proposed method and the deep BSDE method, proposed in \cite{han2018solving}, is the second term in the loss function of \eqref{var_discrete_FBSDE}, corresponding to $\lambda\to\infty$. From \eqref{eq:Var_identity} it becomes clear that, if the driver does not take $Y$ as an input, then this term coincides with the loss function used in the deep BSDE method.

The recent paper \cite{wang2022deep} is, to the best of our knowledge, the first to introduce the variance penalty term in \eqref{var_discrete_FBSDE} for an FBSDE obtained from the dynamic programming principle. The authors of \cite{yu2020backward} also use the variance penalty term, but for general decoupled FBSDEs, \textit{i.e.,} not in the context of stochastic control. In \cite{wang2022deep}, the problem is approached differently in that they have one network for the Markov map for $Z$ and one for the feedback map for the control. In \cite{wang2022deep}, the variance term in the loss function is presented as a measurability loss. Their motivation is that if the BSDE is solved, then the initial state of the $Y-$process is $\F_0-$measurable and hence the variance is zero. Although, one should bear in mind that this is only true for the continuous BSDE. In the discretized version, the $Y-$process is not measurable and no arguments for convergence of the discretization are presented in \cite{wang2022deep}. However, their numerical results are convincing and an error analysis similar to the one presented in Section \ref{sec4} in the present paper could possibly be carried out also in their setting.

Another method which also makes use of the connection to stochastic control (of Hamiltonian systems) was proposed in \cite{ji2021control}. In that paper, the stochastic maximum principle approach to stochastic control was used, which results in a different type of FBSDEs, compared to the one obtained with the dynamic programming principle, that we consider in this paper. More precisely, it is the $Y-$process instead of the $Z-$process which is connected to the control of the SDE. A similarity is that in both papers the algorithms use a method similar to that in \cite{han2016deep}, to include the cost in the objective function. 

Summarizing, only the method in \cite{wang2022deep} is fully comparable to ours, as it is designed to solve the same problem. It has the variance term but not the mean in its loss function. By introducing a loss function that includes both we are able to prove convergence and obtain a robust method.

\subsection{Decoupled FBSDEs and why coupled FBSDEs are important}
It is sometimes claimed in the literature that since coupled FBSDEs can be transformed into decoupled BSDE, it is sufficient to have schemes for the latter, see \textit{e.g.,} \cite{han2020convergence}. Here, we explain this claim and why we, from a practical and application viewpoint, do not agree. If $\psi\colon[0,T]\times\R^d\times\R^\ell\to\R^d$ is sufficiently regular, then it holds by the It\^o formula for $Y_t^\psi=V(t,X_t^\psi)$ and $Z_t^\psi=\sigma^T(t,X_t^\psi)\text{D}_xV(t,X_t^\psi)$ that
\begin{equation*}\begin{dcases}
  X_t^\psi
  =
  x_0+\int_0^t \big(b(s,X_s^\psi,Z_s^\psi)-\psi(s,X_s^\psi,Z_s^\psi)\big)\d s + \int_0^t\sigma(s,X_s^\psi)\d W_s,\\
  Y_t^\psi
  =
  g(X_T^\psi) 
  + 
  \int_t^T 
    \big(f(s,X_s^\psi,Z_s^\psi)-\big\langle (\sigma(s,X_s^\psi)\sigma^T(s,X_s^\psi))^{-1}\sigma(s,X_s^\psi)
    Z_s,\psi(s,X_s^\psi,Z_s^\psi)\big\rangle\big)
  \d s 
 \\\quad\quad\quad -
  \int_t^T \langle Z_s^\psi, \d W_s\rangle,\quad t\in[0,T].
    \end{dcases}
\end{equation*}
Thus taking $\psi$ in such a way that $b-\psi$ does not depend on $Z$, decouples the FBSDE resulting in a BSDE, where the forward equation has no coupling with the backward equation. From this observation, it might be tempting to approximate the optimal Markov map $\sigma^T(t,x)\text{D}_xV(t,x)$ with the deep BSDE method. The problem with this approach is that the deep BSDE method will learn $\sigma^T(t,x)\text{D}_xV(t,x)$ well only around typical trajectories of $X^\psi$, but not around those of $X$. While there is empirical evidence that the deep BSDE method overcomes the curse of dimensionality, it does not at all approximate the solution everywhere, but only around the typical solution trajectories of the forward equation. Since $X$ is controlled, it has a different dynamics than $X^\psi$ and this may ruin the applicability of the deep BSDE method for control problems, if the feedback map is desired. If only an approximation of the solution to the HJB equation (the value) is sought, as in \cite{han2018solving}, then decoupling is feasible.

\section{Fully implementable scheme and neural network regression}\label{sec3}
In this section, we describe how the discrete variational problem \eqref{var_discrete_FBSDE} is approximated with the help of neural network regression. In principle, other function approximators could be used, but neural network regression is arguably one of the most suitable choices due to the ability to approximate complicated high-dimensional functions. Although neural networks have shown empirically high quality results in many different fields, there are still many open convergence questions related to the optimization procedure of the loss function. On the other hand, the Universal Approximation Theorem (UAT) \cite{cybenko1989approximations} guarantees that under certain conditions, there is a neural network, sufficiently deep and wide, such that it is possible to approximate a large class of continuous functions to any, pre-specified degree of accuracy.

\subsection{Fully implementable algorithms}\label{sec_algo}
Without further specifications, \eqref{var_discrete_FBSDE} assumes exact optimization over an unspecified set of functions $\zeta$ and the exact computation of expectations. To define a fully implementable scheme, the Markov maps $\zeta_0,\dots,\zeta_{N-1}$ are approximated with neural networks $\phi_0^{\theta_0},\dots,\phi_{N-1}^{\theta_{N-1}}$ with parameters $\theta=(\theta_0,\dots,\theta_{N-1})$ in some parameter space $\Theta$. We specify them in further detail below. Moreover, expectations are approximated with batch Monte-Carlo simulation. Let $K_{\mathrm{epochs}}\geq1,K_{\mathrm{batch}}\geq1$ be the number of epochs and the number of batches per epoch, respectively. Let further $M_{\mathrm{train}},M_\text{batch}\geq1$ be the size of the training data set and batch, respectively. We assume that $M_{\mathrm{train}}/(2M_\text{batch})=K_{\mathrm{batch}}\in\N$. Training data are $M_{\mathrm{train}}$ independent realizations of the Wiener increments $\Delta W_0,\ldots,\Delta W_{N-1}\sim\mathcal{N}(0,h)$ and the training data are reused in $K_{\mathrm{epoch}}$ epochs. The training is initialized by random sampling of $\theta^0\in\Theta$. For each update step in an epoch of the training algorithm, we take $2M_\text{batch}$ Wiener increments $\Delta W_0(m),\ldots,\Delta W_{N-1}(m)$, $m=1,2,\dots,2M_\text{batch}$ from the training data set that were not previously used during the epoch and update $\theta$ by approximate optimization of the following problem:
\begin{equation}\label{var_discrete_FBSDE_implementable}
\begin{dcases}
\underset{\theta\in\Theta}{\mathrm{minimize}}\ 
 \mathcal{L}_{\lambda,h}(\theta)
 =
 \frac{1}{M_\text{batch}}\Bigg(\sum_{m=1}^{M_\text{batch}}\mathcal{Y}_0^{h,\theta}(m)+\lambda\cdot\sum_{m=M_\text{batch}+1}^{2M_\text{batch}}|g(X_N^{h,\theta}(m))-Y_N^{h,\theta}(m)|^2\Bigg),\quad \text{where},\\
 \mathcal{Y}_0^{h,\theta}(m):=g(X_N^{h,\theta}(m))+\sum_{k=0}^{N-1}f(t_k,X_k^{h,\theta}(m),Z_k^{h,\theta}(m))h 
-\sum_{k=0}^{N-1}\big\langle Z_k^{h,\theta}(m),\Delta W_k(m)\big\rangle,\\
X_n^{h,\theta}(m)=x_0+\sum_{k=0}^{n-1}b(t_k,X_k^{h,\theta}(m),Z_k^{h,\theta}(m))h + \sum_{k=0}^{n-1}\sigma(t_k,X_k^{h,\theta}(m))\Delta W_k(m),\\
Y^{h,\theta}_n(m)=
\frac{1}{M_\text{batch}}\sum_{r=1}^{M_\text{batch}}\mathcal{Y}_0^{h,\theta}(r)
-\sum_{k=0}^{n-1}f(t_k,X^{h,\theta}_k(m),Z_k^{h,\theta}(m))h +\sum_{k=0}^{n-1} \big\langle Z_k^{h,\theta}(m),\Delta W_k(m)\big\rangle,\\
Z_k^{h,\theta}(m)=\phi_k^{\theta_k}(X_k^{h,\theta}(m)).
\end{dcases}
\end{equation}
When all training data has been used, a new epoch starts. When the $K_{\mathrm{epoch}}$:th epoch is finished, the algorithm terminates. The neural network parameters at termination are $\theta^*$. It is an approximation of the parameters $\theta^{**}$ that optimize \eqref{var_discrete_FBSDE_implementable} in the limit $M_\text{batch}\to\infty$. To complement \eqref{var_discrete_FBSDE_implementable}, Algorithm~\ref{alg:osm} details the training procedure. 

It should be noted that the expected value of the stochastic sum in $\mathcal{Y}_0^{h,\theta}$ equals zero. Therefore, the algorithm would also work without it, but a practical reason to keep it is that it decreases the variance of $\mathcal{Y}_0^{h,\theta}$ significantly, and this requires fewer Monte-Carlo samples to achieve the same accuracy. 

\begin{algorithm}[htp]\label{algorithm}
	\KwIn{Initialization of neural network parameters, $\{\theta_0(1),\ldots,\theta_{N-1}(1)\}$, and, for $0\leq k\leq2M_\text{train}$ and $0\leq n\leq N-1$, Wiener increments $\Delta W_n(k)$.}
	\KwOut{Approximation of the Markov map for $(Z_t)_{t\in[0,T]}$ at the time discrete mesh points.}
     \For{$k=1,2,\ldots,K_\mathrm{batch}$ \emph{(}$K_\mathrm{batch}=M_\mathrm{train}/(2M_\mathrm{batch})$ \emph{is the number of batches.)}  \emph{(should be carried out sequentially)}}{
    \For{ $1\leq m\leq M_\mathrm{batch}$ \emph{(may be carried out in parallel)}}{
    Set $X_0^{h,\theta}(m)=x_0$
	
	\For{$n=0, \ldots, N-1$ \emph{(should be done sequentially)}}{
	    $Z_{n}^{h,\theta}(m)=\phi_n\big(X_{n}^{h,\theta}(m)\,|\,\theta_n(k)\big)$\\
		$X_{n+1}^{h,\theta}(m)=X_{n}^{h,\theta}(m) + b\big(t_n,X_{n}^{h,\theta}(m),Z_{n}^{h,\theta}(m)\big)h+ \sigma(t_n,X_{n}^{h,\theta}(m))\Delta W_n(m)$
	}}
	\vspace{0.2cm}

\For{ $m\in\{M_\mathrm{batch}+1,\ldots,2M_\mathrm{batch}\}$ \emph{(may be carried out in parallel)}}{
	\vspace{0.2cm}
	Set $X_0^{h,\theta}(m)=x_0$ and $Y_0^{h,\theta}(m)=\frac{\displaystyle 1}{\displaystyle M_\mathrm{batch}}\displaystyle \sum_{m=M_\mathrm{batch}+1}^{2M_\mathrm{batch}} g(X_N^{h,\theta}(m))+ \sum_{n=0}^{N-1}f\big(t_n,X_{n}^{h,\theta}(m),Z_{n}^{h,\theta}(m)\big)h- \sum_{n=0}^{N-1}\big\langle Z_n^{h,\theta}(m),\Delta W_n(m)\big\rangle$ \\[1.5ex]
	\For{$n=0, \ldots, N-1$ \emph{(should be carried out sequentially)}}{
	$Z_{n}^{h,\theta}(m)=\phi_n\big(X_{n}^{h,\theta}(m)\,|\,\theta_n(k)\big)$\\
	$X_{n+1}^{h,\theta}(m)=X_{n}^{h,\theta}(m) + b\big(t_n,X_{n}^{h,\theta}(m),Z_{n}^{h,\theta}(m)\big)h+ \sigma(t_n,X_{n}^{h,\theta}(m))\Delta W_n(m)$\\
    $Y_{n+1}^{h,\theta}(m)=Y_{n}^{h,\theta}(m) - f\big(t_n,X_{n}^{h,\theta}(m),Z_{n}^{h,\theta}(m)\big)h+ \big\langle Z_{n}^{h,\theta}(m),\Delta W_n(m)\big\rangle$
	}}    \vspace{0.2cm}
	$\theta=\{\theta_0,\theta_1,\ldots,\theta_{N-1}\}$ \text{(trainable parameters)}\\    \vspace{0.2cm}
	$\mathcal{L}(\theta)=\frac{1}{M_\text{batch}}\bigg(\sum_{m=1}^{M_\mathrm{batch}}Y_0^\theta(m)+\lambda\sum_{m=M_\mathrm{batch}+1}^{2M_\mathrm{batch}}|g(X_N^{h,\theta}(m))-Y_N^{h,\theta}(m)|^2$\bigg) (Loss-function)\\
    \vspace{0.2cm}
	$\theta(k+1)\leftarrow\argmin_{\theta}\mathcal{L}(\Theta)$ (some optimization algorithm, usually of gradient decent type)
}
\caption{Pseudo-code of one epoch of the neural network training}
\label{alg:osm}
\end{algorithm}

\subsection{Specification of the neural networks}\label{NN_spec}
Here, we introduce the neural networks that we use in our implementations in Section~\ref{sec5}. The generality is kept to a minimum and more general architectures are of course possible. For each $\phi_k^{\theta_k}\colon\R^d\to\R^\ell$, a fully connected neural network with two hidden layers with 20 nodes in each layer and a ReLU activation function $\mathfrak{R}(x)=\max(0,x)$ acting elementwise is employed. More precisely, the $\phi_k^\theta$ is of the form
\begin{align*}
    \phi_k^{\theta_k}(x)
    =
    W_k^3\mathfrak{R}(W_k^2\mathfrak{R}(W_k^1x+b_k^1)+b_k^2)+b_k^3,
\end{align*}
with weight matrices 
$W_k^1\in\R^{20\times d}$,  $W_k^2\in\R^{20\times 20}$, $W_k^3\in\R^{\ell\times 20}$
and bias vectors
$b_k^1,b_k^2\in\R^{\mathfrak{N}}$,
 $b_k^3\in\R^\ell$,
and $\theta_k=(W_k^1,W_k^2,W_k^3,b_k^1,b_k^2,b_k^3)$, where the matrices are considered vectorized before concatenation. 

\section{Convergence analysis}\label{sec4}
In this section, we primarily provide an error analysis for \eqref{var_discrete_FBSDE}, i.e., for the semidiscretization in time. In Section~\ref{sec:notation}, we introduce notation and spaces, and, in Section~\ref{sec:setting}, we present the setting and some further notation. Two technical results on strong and weak convergence  are stated and proved in Section~\ref{sec:aux}. These two results are used in Section~\ref{sec:Y0} to prove the error in the objective function, in the initial and terminal value of $Y$ and for the variance of the stochastic cost. The results hold under an assumption on the regularity of the exact continuous and discrete Markov maps. Convergence of the latter to the former is not assumed.
Section~\ref{sec:strong} contains strong convergence of the FBSDE under either the stronger assumption of small time $T$ or convergence of the discrete Markov maps. A discussion about a full error analysis for the fully implementable scheme \eqref{var_discrete_FBSDE_implementable} is presented in Section~\ref{sec:full}.

\subsection{Notation and spaces}\label{sec:notation}
For Euclidean spaces $\R^k$, $k\geq1$, we denote by $\|\cdot\|$ the 2-norm without specifying the dimension. Let $\mathcal{S}^2(\R^k)$ and $\mathcal{H}^2(\R^k)$ be the spaces of all progressively measurable stochastic processes $y,z\colon[0,T]\times\Omega\to\R^k$, for which
\begin{align*}
    \|y\|_{\mathcal{S}^2(\R^k)}
    =
    \sup_{t\in[0,T]}
    \Big(
    \E
    \big[
      \|y_t\|^2
    \big]
    \Big)^\frac12
    <
    \infty,
    \quad
    \textrm{and}
    \quad
    \|z\|_{\mathcal{H}^2(\R^k)}
    =
    \Bigg(
    \E
    \bigg[
      \int_0^T
        \|z_t\|^2
      \d t
    \bigg]
    \Bigg)^\frac12<\infty,
\end{align*}
respectively. For a discretization, $0=t_0<t_1<\dots<t_N=T$ with $t_{n+1}-t_n=h$, for all $n$, the space $\mathcal{S}_h(\R^k)$ is the space of all $\mathcal{F}^h$-adapted, square integrable and discrete stochastic processes $y:\{0,1,\dots,N\}\times\Omega\to\R^d$, where $\mathcal{F}^h_n=\sigma(\Delta W_m,\ m=0,\dots,n-1)$. For $k_1,k_2,k_3,k_4\geq1$,  $\ell_1,\ell_2,\ell_3\geq0$ and regular $S_i\subseteq\R^{k_i}$, $i=1,2,3,4$, by $C_{\mathrm b}^{\ell_1,\ell_2,\ell_3}(S_1\times S_2\times S_3;S_4)$, we denote the space of all functions $\phi\colon S_1\times S_2\times S_3\to S_4$, whose derivatives up to orders $\ell_1,\ell_2,\ell_3$ exist, are continuous and bounded. We equip it with the semi-norms 
\begin{align*}
    |\phi|_{\gamma}
    =
    \sup_{x\in S_1\times S_2\times S_3}
    \|\partial^\gamma\phi(x)\|,
    \quad
    i\in\{1,\dots,\ell\},
\end{align*}
where $\gamma=\{(i_1,i_2,i_3): i_j\in\{0,,\dots,\ell_j\},\ j=1,2,3\}$ are multi-indices and $\partial^{\gamma}=\partial_1^{i_1}\partial_2^{i_2}\partial_3^{i_3}$ with $\partial_j^i$ denoting $i$:th partial derivative in variable $j$. The set $B(\ell_1,\ell_2,\ell_3)$ denotes all multi-indices of length 3 that have exactly one non-zero index. We only use multi-indices in $B(\ell_1,\ell_2,\ell_3)$ and do therefore not impose restrictions on the cross derivatives, as is otherwise common. For functions with fewer than three variables, we reduce the number of indices accordingly and in the semi-norms we write  $|\phi|_{i_1,i_2,i_3}=|\phi|_{(i_1,i_2,i_3)}$. For $\alpha\in(0,1]$, $k_1,k_2,k_3\geq1$, $\ell_1,\ell_2\geq0$ and regular $S_i\subseteq \R^{k_i}$, $i=1,2,3$, by $C_{\mathrm{H}, \mathrm b}^{\alpha,\ell_1,\ell_2}([0,T]\times S_1\times S_2;S_3)$, we denote the space of all functions $\phi\colon [0,T]\times S_1\times S_2\to S_3$ that are $\alpha$-H\"older continuous in time and whose derivatives of order $\ell_1,\ell_2$ exist, are continuous and bounded and satisfy the property 
\begin{align*}
    \triple\phi\triple_{\alpha,\ell_1,\ell_2}
    &=
    \sup_{t\in[0,T]}
    \|\phi(t,0,0)\|
    +
    \sum_{\gamma \in B}
      \sup_{t\in[0,T]}|\phi(t,\cdot,\cdot)|_\gamma\\
    &\quad
    +
    \sup_{(x_1,x_2)\in S_1\times S_2}
    \sup_{t_1,t_2\in[0,T],t_1\neq t_2}\frac{\mathbf{1}_{(0,1]}(\alpha)\|\phi(t_1,x_1,x_2)-\phi(t_2,x_1,x_2)\|}
    {(1+\|x_1\|+\|x_2\|)|t_2-t_1|^\alpha}
    <\infty.
\end{align*}
Again, when there is only one space variable, we reduce the number of indices. When there is no risk of confusion, we write $|\cdot|_{\alpha,0,0}$ to denote the second term defining the $\triple\cdot\triple_{\alpha,\ell_1,\ell_2}$-norm, and let $|\cdot|_{0,i,0}$ and $|\cdot|_{0,0,i}$ coincide with the semi-norms on $C_{\mathrm b}^{\alpha,\ell_1,\ell_2}([0,T]\times S_1\times S_2;S_3)$ with the same notation. For $\alpha=0$, we let
$C_{\mathrm{H}, \mathrm b}^{0,\ell_1,\ell_2}([0,T]\times S_1\times S_2;S_3)=C_{\mathrm b}^{0,\ell_1,\ell_2}([0,T]\times S_1\times S_2;S_3)$. 

For any function or process $R$ defined on $[0,T]$, we denote by $\check{R}$ the discrete time function or process defined by $\check{R}_n=R_{t_n}$, $n=0,\dots,N$. For any discrete function or process $R_h$ defined on ${0,\dots,N}$, we write $\hat R_h$ for the continuous time function or process defined by the piecewise constant interpolation  $\hat{R}_{h,t}=R_{h,n}$ for $t\in[t_n,t_{n+1})$ and $\hat{R}_{h,T}=R_{h,N}$.

\subsection{Setting and spaces of Markov maps}
\label{sec:setting}
Let $(W_t)_{t\in[0,T]}$ be a $k$-dimensional Brownian motion on a filtered probability space $(\Omega,\mathcal{F},(\mathcal{F}_t)_{t\in[0,T]},\P)$, and $(\alpha,\beta)\in[0,\frac12]\times\{1\}$ or $(\alpha,\beta)=(1,2)$. The coefficients $ b\colon[0,T]\times\R^d\times \R^k\to\R^d $, $\sigma\colon[0,T]\times\R^d\to\R^{d\times k}$,  $f\colon[0,T]\times\R^d\times \R^k\to\R $ and $g\colon\R^d\to\R$ satisfy 
$b\in C_{\mathrm{ H,b}}^{\alpha,\beta,\beta}([0,T]\times\R^d\times \R^k;\R^d)$, $\sigma\in C_{\mathrm{ H,b}}^{\alpha,\beta}([0,T]\times\R^d;\R^{d\times k})$, $f\in C_{\mathrm{ H,b}}^{\alpha,\beta,\beta}([0,T]\times\R^d\times \R^k;\R)$,
$g\in C^1_{\mathrm b}(\R^d;\R)$. In the case $\alpha=1$, we assume that $\text{D}_x \sigma=0$. 

We next introduce families of Markov maps. Let $\mathcal{Z}=\mathcal{Z}(b,\sigma,f,g)$ be the collection of all measurable functions $\zeta:[0,T]\times\R^d\to\R^k$ with the property that the stochastic processes\\ $(X^\zeta,Y^\zeta,Z^\zeta)_{\zeta\in\mathcal{Z}}\subset\mathcal{S}^2(\R^d)\times\mathcal{S}^2(\R)\times\mathcal{H}^2(\R^d)$ satisfying for all $t\in[0,T]$, $\P$-almost surely \eqref{var_FBSDE}, are well defined. We write $\mathcal{Z}^{\alpha,\beta}=\mathcal{Z}\cap C_{\mathrm{H,b}}^{\alpha,\beta}([0,T]\times\R^d;\R^k)$. For the discrete equations, we introduce for every $h\in(0,1)$ analogously  $\mathcal{Z}_h=\mathcal{Z}_h(b,\sigma,f,g)$ to be the collection of all measurable functions $\zeta\colon\{0,\dots,N_h-1\}\times\R^d\to\R^k$, with the property that $(X^{h,\zeta},Y^{h,\zeta},Z^{h,\zeta})_{\zeta\in\mathcal{Z}_h}\subset\mathcal{S}_h^2(\R^d)\times\mathcal{S}_h^2(\R)\times\mathcal{S}_h^2(\R)$ satisfying for all $n\in\{0,\dots,N_h\}$, $\P$-almost surely \eqref{var_discrete_FBSDE} are well-defined. We write $\mathcal{Z}_h^\beta=\mathcal{Z}_h\cap (C_{\mathrm{b}}^\beta(\R^d;\R^k))^{N+1}$. 

By introducing assumptions on the Markov maps we eliminate the need for assuming smoothness, Lipschitz, polynomial growth, monotonicity, coercivity or other conditions on the coefficients, for the existence and uniqueness of solutions for \eqref{var_FBSDE} and \eqref{var_discrete_FBSDE}. For unfortunate choices of $b,\sigma,f,g$, the spaces $\mathcal Z$ and $\mathcal{Z}_h$ might be empty and results hold by default, but given regular $b,\sigma,f,g$ it is not hard, using available solution theory, to find $\zeta\in\mathcal Z$ and $\zeta_h\in\mathcal{Z}_h$. Still, the entire spaces might be hard to represent but this is not of central importance. In Assumption~\ref{as:Y0} below, the regularities of the optimal $\zeta^*$ and $\zeta_h^*$, solving \eqref{var_FBSDE} and \eqref{var_discrete_FBSDE} are though of importance. Thus classical solution theory for SDE, FBSDE and regularity theory for optimal Markov maps in discrete and continuous time are required for verifying our assumptions for concrete examples. The latter is not well developed, see the discussion prior to Assumption~\ref{as:Y0} below. Thus part of our assumptions require further theoretical development to be verified, but our results show what is required.

\subsection{Auxiliary lemmata on strong and weak convergence for SDEs}\label{sec:aux}

In the proof of our convergence results in Subsection~\ref{sec:Y0}, we rely on the  strong convergence result, stated next. It contains both classical strong convergence of the Euler-Maruyama scheme for H\"older continuous coefficients, including strong order 1 for additive noise, but also a non-standard type of strong convergence result for processes that have drift coefficients that for each step size, agree between the grid points, but whose coefficients do not necessarily converge as the step size tends to zero.
\begin{lemma}\label{lemma_strong_conv}
Suppose the setting of Subsection~\ref{sec:setting} holds. Let $a\in C_{\mathrm{H,b}}^{\alpha,\beta}([0,T]\times\R^d;\R^d)$ and $a_h\colon[0,T]\times\R^d\to\R^d$, $h\in(0,1)$ be a family of functions that are constant on each interval $[t_n,t_{n+1})$, satisfy $a_h([0,T],\cdot)\subset C^\beta(\R^d;\R^d)$ and $\sup_{h\in(0,1)}\triple a_h\triple_{0,\beta}<\infty$, let $\mathcal{X},\mathcal{X}^{1,h}\in\mathcal{S}^2(\R^d)$, $h\in(0,1)$ be the unique solutions to
\begin{align*}
    \emph{d} \mathcal{X}_t
    &=
    a(t,\mathcal{X}_t)\emph{d} t
    +
    \sigma(t,\mathcal{X}_t)\emph{d} W_t,\ t\in(0,T];\quad \mathcal{X}_0 = x_0,\\
    \emph{d} \mathcal{X}_t^{1,h}
    &=
    a_h(t,\mathcal{X}_t^{1,h})\emph{d} t
    +
    \sigma(t,\mathcal{X}_t^{1,h})\emph{d} W_t,\ t\in(0,T];\quad \mathcal{X}_0^{1,h} = x_0,\quad h\in(0,1),
\end{align*}
and $\mathcal{X}^{2,h},\mathcal{X}^{3,h}\in\mathcal{S}_h^2(\R^d)$, $h\in(0,1)$ be the unique solutions to
\begin{align*}
    \mathcal{X}_{n+1}^{2,h}
    &=
    \mathcal{X}_{n}^{2,h}
    +
    a(t_n,\mathcal{X}_n^{2,h})h
    +
    \sigma(t_n,\mathcal{X}_n^{2,h})\Delta W_n,\ n\in\{0,\dots,N-1\};\quad \mathcal{X}_0^{2,h} = x_0,\quad h\in(0,1),\\
    \mathcal{X}_{n+1}^{3,h}
    &=
    \mathcal{X}_{n}^{3,h}
    +
    a_h(t_n,\mathcal{X}_n^{3,h})h
    +
    \sigma(t_n,\mathcal{X}_n^{3,h})\Delta W_n,\ n\in\{0,\dots,N-1\};\quad \mathcal{X}_0^{3,h} = x_0,\quad h\in(0,1).
\end{align*}
It holds that 
\begin{align}\label{eq:stability}
  \sup_{h\in(0,1)}
  \|\mathcal{X}^{1,h}\|_{\mathcal{S}^2(\R^d)}
  +
    \sup_{h\in(0,1)}
  \|\hat{\mathcal{X}}^{2,h}\|_{\mathcal{S}^2(\R^d)}
  +
  \sup_{h\in(0,1)}
  \|\hat{\mathcal{X}}^{3,h}\|_{\mathcal{S}^2(\R^d)}
  <\infty,
\end{align}
and there exists a constant $C$, such that
\begin{equation*}
    \lim_{h\to0}
    h^{-\alpha}
    \Big(
    \big\|
      \mathcal{X}-\hat{\mathcal{X}}^{2,h}\big
    \|_{\mathcal{S}^2(\R^d)}
    +
    \big\|
      \mathcal{X}^{1,h}-\hat{\mathcal{X}}^{3,h}
    \big\|_{\mathcal{S}^2(\R^d)}
    \Big)
    =
    \begin{cases}
      0\quad&\textrm{if } \alpha=0,\\
      C&\textrm{otherwise}.
    \end{cases}
\end{equation*}
\end{lemma}

\begin{proof}
Due to the assumption on $a_h$, there exists a finite constant $C$ such that
\begin{align*}
    \|a_h(t,x)\|
    \leq
    \sup_{h\in(0,1)}
    \sup_{t\in[0,T]}
    \big(
      \|a_h(t,0)\|+|a_h(t,\cdot)|_1
    \big)
    \|x\|
    \leq C\|x\|.
\end{align*}
This is enough to show the first assertion, the stability estimate, using a Gronwall argument. For the convergence, we rely on \cite{kruse2012characterization}, which covers both the cases $\alpha\in[0,\frac12]$ and $\alpha=1$. It is based on the fact that bistability and consistency of order $\alpha$, in the sense of \cite{kruse2012characterization}, imply convergence of order $\alpha$. We start with $\mathcal{X}^{1,h}\to\mathcal{X}^{3,h}$ and the convergence $\mathcal{X}^{2,h}\to\mathcal{X}$ is immediate afterwards. Since the drift coefficients of $\mathcal{X}^{1,h}$ and $\mathcal{X}^{3,h}$ are both $h$-dependent, are not assumed to converge, what we want to prove, i.e., 
$\lim_{h\to0}\|\mathcal{X}_t^{1,h}-\hat{\mathcal{X}}_t^{3,h}\|_{\mathcal{S}^2(\R^d)} =0$ does not say anything about convergence of $\mathcal{X}^{3,h}$ and $\mathcal{X}^{1,h}$. In particular, $\mathcal{X}_t^{1,h_1}-\hat{\mathcal{X}}_t^{3,h_2}$ is in general not small for $h_1\neq h_2$. Due to this special setting, we have to verify that the proofs of \cite{kruse2012characterization} are still valid.

We start with proving bistability. First, in \cite[Lemma 4.1]{kruse2012characterization}, it is proven that for small enough $h$ the numerical scheme is bijective. For our proof, it is important that the upper bound for this property does not depend on $h$. From the proof, it is clear that this bound depends on the reciprocal of $\sup_{h\in(0,1)}\sup_{t\in[0,T]}|a_h(t,\cdot)|_1$ and is therefore bounded away from zero by assumption. By following the proof of \cite[Lemma 4.2]{kruse2012characterization} line by line, the bistability constants are bounded by $\sup_{h\in(0,1)}\sup_{t\in[0,T]}|a_h(t,\cdot)|_1$ and its reciprocal when Assumptions~(S1) and (S2) in \cite{kruse2012characterization} hold. It remains to prove (S1) and (S2) with uniform constants. First, (S1) is trivially satisfied with $L=0$ for any explicit scheme. By inspection of the proof of \cite[Theorem 3.3]{kruse2012characterization}, we see that, for $\alpha\in[0,\frac12]$, corresponding to the stochastic $\theta$-method with $\theta=0$, (S2) is satisfied for all $h$ with $L=3T\sup_{h\in(0,1)}\sup_{t\in[0,T]}|a_h(t,\cdot)|_1^2+12d|\sigma|_1^2$. For $\alpha=1$, (S2) holds for the same $L$ as the schemes are the same. This concludes the proof of bistability. 

Consistency for $\alpha\in[0,\frac12]$ is obtained by observing that the proof in \cite{kruse2012characterization}, referring for details to \cite{beyn2010two}, only relies on \eqref{eq:stability} and the $\frac12$-H\"older continuity of $a_h$ between the grid points, and this is clearly satisfied since $a_h$ is constant on $[t_n,t_{n+1})$. For consistency in case $\alpha=1$, we rely on the proof of consistency for the It\^o-Taylor scheme of order one, i.e., the Milstein scheme, coinciding with our scheme under the assumption of additive noise. It suffices to note that the analysis is conducted per interval, and to check that $f_\alpha$ is globally Lipschitz continuous for $\alpha\in\mathcal B(\mathcal{A}_1)$, in the notation of \cite{kruse2012characterization}. For this to hold, it suffices that $\sup_{h\in(0,1)}\sup_{t\in[0,T]}|a_h(t,\cdot)|_2<\infty$. By assumption, this completes the proof for $\mathcal{X}^{1,h}$ and $\mathcal{X}^{3,h}$. For $\mathcal{X}^{2,h}$ and $\mathcal{X}$, the convergence is included in \cite{kruse2012characterization}, up to the order of consistency, which in our setting is lower when $\alpha\in[0,\frac12)$. It is straight forward to use $\alpha$-H\"older continuity for $\alpha\in(0,\frac12)$, exactly as in the case $\alpha=\frac12$, and this way get consistency of order $\alpha$. For $\alpha=0$, one uses dominated convergence instead to obtain the convergence without order. This is valid under our assumptions.
\end{proof}

Our next Lemma is a weak convergence type result. With our assumption of low regularity on the FBSDE coefficients, the weak rate does not exceed the strong rate and therefore the proof relies on the strong convergence results of Lemma~\ref{lemma_strong_conv}. Obtaining weak convergence order $\alpha\in(\frac12,1)$ for multiplicative noise, would otherwise require a H\"older condition on the third derivative of $b,\sigma,f,g$ and the case $\alpha=1$ requires four derivatives \cite{mikulevicius1991rate,kohatsu2017weak}.

\begin{lemma}\label{weak_conv_value_fun}
Suppose the setting of Subsection~\ref{sec:setting} holds.
For all functions $\zeta\in \mathcal{Z}^{\alpha,\beta}$, collection of functions $\zeta_h\in \mathcal{Z}_h^\beta$, $h\in(0,1)$, satisfying
$
    \sup_{h\in(0,1)}
    \triple
      \hat{\zeta}_h
    \triple_{0,\beta}
    <\infty,
$ and $\lambda\geq0$, there exists a constant $C$, such that
\begin{align*}
     \lim_{h\to0}
     h^{-\alpha}
     \Big(
       |\Phi_\lambda(\zeta)-\Phi_{\lambda,h}(\check{\zeta})|
       +
       |\Phi_\lambda(\hat{\zeta}_h)-\Phi_{\lambda,h}(\zeta_h)|
     \Big)
     =  
    \begin{dcases}
      0\quad&\textrm{if } \alpha=0,\\
      C&\textrm{otherwise}.
    \end{dcases}
\end{align*}
\begin{proof}
We start with the first term and observe that 
\begin{align}\label{eq:weak_split}
    |\Phi_\lambda(\zeta)-\Phi_{\lambda,h}(\check{\zeta})|
    \leq
    \big|\E\big[\mathcal{Y}_0^{\zeta}
      -
     \mathcal{Y}_0^{h,\check{\zeta}}\big]\big|
    +
    \lambda
    \big|\mathrm{Var}\big(\mathcal{Y}_0^{\zeta}\big)
      -
    \mathrm{Var}\big(\mathcal{Y}_0^{h,\check{\zeta}}\big)\big|.
\end{align}
For the proof of both terms of \eqref{eq:weak_split}, we rely on convergence of $(\E\big[(\mathcal{Y}_0^{\zeta}-\mathcal{Y}_0^{h,\check{\zeta}})^2\big])^\frac12$, which we next prove, beginning with $\alpha>0$. For $t\in[t_n,t_{n+1})$, denote $\bar t=t_n$. We start by observing that by the definition,
\begin{align*}
    \mathcal{Y}_0^{\zeta}-\mathcal{Y}_0^{h,\check{\zeta}}
     &=g(X_T^\zeta)-g(X_T^{h,\check\zeta})
     +
     \int_0^T\big(f\big(t,X_t^\zeta,\zeta(t,X_t^\zeta)\big)-f\big(\bar{t},\hat X_t^{h,\check\zeta},\zeta(\bar{t},\hat X_t^{h,\check\zeta})\big)\big)\d t\\
     &\quad+
     \int_0^T\big\langle\zeta(t,X_t^\zeta)-\zeta(\bar{t},\hat{X}_t^{h,\check\zeta}),\d W_t\big\rangle\\
     &=
     g(X_T^\zeta)-g(X_T^{h,\check\zeta})
     +
     \int_0^T\big(f\big(t,X_t^\zeta,\zeta(t,X_t^\zeta)\big)-f\big(t,\hat X_t^{h,\check\zeta},\zeta(t,\hat X_t^{h,\check\zeta})\big)\big)\d t\\
     &
     \quad+
     \int_0^T\big(f\big(t,\hat X_t^{h,\check\zeta},\zeta(t,\hat X_t^{h,\check\zeta})\big)-f\big(\bar{t},\hat X_t^{h,\check\zeta},\zeta(\bar{t},\hat X_t^{h,\check\zeta})\big)\big)\d t
     +\int_0^T\big\langle\zeta(t,X_t^\zeta)-\zeta(t,\hat{X}_t^{h,\check\zeta}),\d W_t\big\rangle
    \\
     &\quad +\int_0^T\big\langle\zeta(t,X_t^\zeta)-\zeta(t,\hat{X}_t^{h,\check\zeta}),\d W_t\big\rangle+\int_0^T\big\langle\zeta(t,\hat{X}_t^{h,\check\zeta})-\zeta(\bar{t},\hat{X}_t^{h,\check\zeta}),\d W_t\big\rangle
\end{align*}
From the It\^o Isometry and the triangle inequality, we have
\begin{align*}
    \big(
    \E
    \big[
      (\mathcal{Y}_0^{\zeta}-\mathcal{Y}_0^{h,\check{\zeta}})^2
    \big]
    \big)^\frac12
     &\leq
     \big(
     \E 
     \big[
       (g(X_T^\zeta)-g(X_T^{h,\check\zeta}))^2
     \big]
     \big)^\frac12
     +
     \int_0^T
       \big(
         \E\big[
           \big(f\big(t,X_t^\zeta,\zeta(t,X_t^\zeta)\big)-f\big(t,\hat X_t^{h,\check\zeta},\zeta(t,\hat X_t^{h,\check\zeta})\big)\big)^2\big]\big)^\frac12\d t\\
     &
     \quad+
     \int_0^T
       \E\big[
         \big(f\big(t,\hat X_t^{h,\check\zeta},\zeta(t,\hat X_t^{h,\check\zeta})\big)-f\big(\bar{t},\hat X_t^{h,\check\zeta},\zeta(\bar{t},\hat X_t^{h,\check\zeta})\big)\big)^2
       \big]\big)^\frac12\d t\\
     &\quad+
     \Big(
       \int_0^T
         \E\big[
           \big\| \zeta(t,X_t^\zeta)-\zeta(t,\hat X_t^{h,\check\zeta})
           \big\|^2
         \big]\d t\Big)^\frac12
    +
    \Big(
       \int_0^T
         \E\big[
           \big\|\zeta(t,\hat X_t^{h,\check\zeta})-\zeta(\bar t,\hat X_t^{h,\check\zeta})
           \big\|^2
         \big]\d t\Big)^\frac12\\
    &= I_1+I_2+I_3+I_4+I_5.
\end{align*}

For $I_1$, we use Lipschitz continuity of $g$ and the Cauchy-Schwarz inequality, to get
\begin{align}\label{eq:I1}
    I_1
    \leq
    |g|_1\E\big[\|X_T^\zeta-X_T^{h,\check\zeta}\|\big]
    \leq
    |g|_1\big(\E\big[\|X_T^\zeta-X_T^{h,\check\zeta}\|^2\big]\big)^\frac12
    \leq
    |g|_1 \big\|X^\zeta-\hat X^{h,\check\zeta}\big\|_{\mathcal{S}^2(\R^d)}.
\end{align}
The function $\phi(t,x)=f(t,x,\zeta(t,x))$ is uniformly Lipschitz continuous in $x$ with Lipschitz constant
$C_\phi=|f|_{0,1,0}+|f|_{0,0,1}|\zeta|_{0,1}$.
This implies, together with the Cauchy-Schwarz inequality,
\begin{align}\label{eq:I2}
    I_2
    \leq
    C_\phi
    \int_0^T
       \big(\E\big[\|X_t^\zeta-\hat X_t^{h,\check\zeta}\|^2\big]\big)^\frac12
    \d t
    \leq
    C_\phi T
       \big\|X^\zeta-\hat X^{h,\check\zeta}\big\|_{\mathcal{S}^2(\R^d)}.
\end{align}
Using the assumptions of $b$ and $\zeta$, it is easily verified that $\phi$ belongs to $C_{\mathrm{H,b}}^{\alpha,\beta}([0,T]\times\R^d;\R^d)$ and therefore, by \eqref{eq:I1}, \eqref{eq:I2} and Lemma~\ref{lemma_strong_conv}, we have $I_1+I_2\leq C h^\alpha$. For the term $I_3$, we see that
\begin{align*}
    I_3
    &\leq
    \Big(|f|_{\alpha,0,0} + |f|_{0,0,1}|\zeta|_{\alpha,0}\Big)
    \Bigg(1+\sup_{h\in(0,1)}
    \|\hat X^{h,\check\zeta}\|_{\mathcal{S}^2(\R^d)}
    \Bigg)
    \sum_{n=0}^{N-1}
    \int_{t_n}^{t_{n+1}}(t- t_n)^\alpha \d t\\
    &=
    \frac{|f|_{\alpha,0,0} + |f|_{0,0,1}|\zeta|_{\alpha,0}}{1+\alpha}
    \Bigg(1+
      \sup_{h\in(0,1)}
      \|\hat X^{h,\check\zeta}\|_{\mathcal{S}^2(\R^d)}
    \Bigg)
    N h^{1+\alpha}
    \leq
    C h^\alpha.
\end{align*}
Similarly, for $I_4$ and $I_5$ we have
\begin{align*}
    I_4
    \leq
    |\zeta|_{0,1}
    \Big(
      \int_0^T
        \E
        \big[
           \| 
             X_t^\zeta-\hat X_t^{h,\check\zeta}
           \|^2
        \big]
      \d t
    \Big)^\frac12
    \leq
    |\zeta|_{0,1}
    T^\frac12
    \|
      X^\zeta-\hat X^{h,\check\zeta}
    \|_{\mathcal{S}^2(\R^d)}
    \leq
    C h^\alpha,
\end{align*}
and
\begin{align*}
    I_5
    &\leq
    |\zeta|_{\alpha,0}
    \Big(
      1+\sup_{h\in(0,1)}\|\hat{X}^{h,\check\zeta}\|_{\mathcal{S}^2(\R^d)}
    \Big)
    \Bigg(
       \sum_{n=0}^{N-1}
       \int_{t_n}^{t_{n+1}}
         |t-t_n|^{2\alpha}\d t
    \Bigg)^\frac12\\
    &=
    \frac{|\zeta|_{\alpha,0}}{1+2\alpha}
    \Big(
      1+\sup_{h\in(0,1)}\|\hat{X}^{h,\check\zeta}\|_{\mathcal{S}^2(\R^d)}
    \Big)
    \big(
       N
       h^{1+2\alpha}
    \big)^\frac12   
    \leq C h^\alpha.
\end{align*}

For $\alpha=0$, the terms $I_1, I_2, I_4$ need no special treatment. For $I_3$, we notice that the function $\phi$ is only continuous in $t$ and convergence without rate holds by the Dominated Convergence Theorem. This is verified by noting that 
\begin{align*}
    &\sup_{t\in[0,T]}
    \Big(\E\Big[\Big(f\big(t,\hat{X}_t^{h,\check\zeta},\zeta(t,\hat{X}_t^{h,\check\zeta})\big)-f\big(\bar{t},\hat{X}_t^{h,\check\zeta},\zeta(\bar{t},\hat{X}_t^{h,\check\zeta})\big)\Big)^2\Big]\Big)^\frac12
    \leq
    2\sup_{t\in[0,T]}
    \Big(\E\Big[\big(f\big(t,\hat{X}_t^{h,\check\zeta},\zeta(t,\hat{X}_t^{h,\check\zeta})\big)^2\Big]\Big)^\frac12\\
    &\qquad
    \leq
    2\sup_{t\in[0,T]}
    \Big(\E\Big[\big(f\big(t,\hat{X}_t^{h,\check\zeta},\zeta(t,\hat{X}_t^{h,\check\zeta})-f(t,0,0)\big)^2\Big]\Big)^\frac12
    +
    2\sup_{t\in[0,T]}
    |f(t,0,0)|\\
    &\qquad
    \leq
    2\sup_{t\in[0,T]}
    \big(
      |f(t,\cdot,\cdot)|_{1,0}+|f(t,\cdot,\cdot)|_{0,1}
      \big(
        |\zeta(t,\cdot)|_1
        +
        \|\zeta(t,0)\|
      \big)
    \big)
    \sup_{h\in(0,1)}\|\hat{X}^{h,\check\zeta}\|_{\mathcal{S}^2(\R^d)}
    +
    2\sup_{t\in[0,T]}
    \big|\E\big[f\big(t,0,0\big)\big|.    
\end{align*}
Due to the assumptions, the right-hand side is finite. The term $I_5$ admits a similar treatment and we refrain from giving details. This proves that 
\begin{align}\label{eq:y0}
    \lim_{h\to0}
    \big(
    \E
    \big[
      (\mathcal{Y}_0^{\zeta}-\mathcal{Y}_0^{h,\check{\zeta}})^2
    \big]
    \big)^\frac12
    =
    \begin{dcases}
      0\quad&\textrm{if } \alpha=0,\\
      C&\textrm{otherwise}.
    \end{dcases}
\end{align}
For the first term in \eqref{eq:weak_split}, this and the fact that 
$    
    \big|\E\big[\mathcal{Y}_0^{\zeta}
      -
     \mathcal{Y}_0^{h,\check{\zeta}}\big]\big|
     \leq
     \big(
    \E
    \big[
      (\mathcal{Y}_0^{\zeta}-\mathcal{Y}_0^{h,\check{\zeta}})^2
    \big]
    \big)^\frac12
$ prove the convergence. 

For the second term in \eqref{eq:weak_split}, we use the conjugate rule, Cauchy-Schwarz' inequality, the triangle inequality, to get
\begin{align*}
    \big|
      \mathrm{Var}
      (\mathcal{Y}_0^\gamma)
      -
      \mathrm{Var}
      (\mathcal{Y}_0^{h,\check\gamma})
    \big|
&   =
    \Big|
      \E
      \Big[\big|\E[\mathcal{Y}_0^\gamma]-\mathcal{Y}_0^\gamma\big|^2
      -
      \big|\E\big[\mathcal{Y}_0^{h,\check\gamma}\big]-\mathcal{Y}_0^{h,\check\gamma}\big|^2\Big]
    \Big|\\
&   =
    \big|
      \E\big[
        \big(
          \E[\mathcal{Y}_0^\gamma-\mathcal{Y}_0^{h,\check\gamma}\big]
        +\mathcal{Y}_0^{h,\check\gamma}-\mathcal{Y}_0^\gamma
        \big)
        \big(\E[\mathcal{Y}_0^\gamma+\mathcal{Y}_0^{h,\check\gamma}\big]
        -\mathcal{Y}_0^{h,\check\gamma}-\mathcal{Y}_0^\gamma
        \big)
      \big]
    \big|\\
&   \leq
    \Big(
      \E\Big[
        \big(
          \E[\mathcal{Y}_0^\gamma-\mathcal{Y}_0^{h,\check\gamma}\big]
        +\mathcal{Y}_0^{h,\check\gamma}-\mathcal{Y}_0^\gamma
        \big)^2
      \Big]
    \Big)^\frac12
    \Big(
      \E\Big[
        \big(
        \E\big[
          \mathcal{Y}_0^\gamma+\mathcal{Y}_0^{h,\check\gamma}\big]
        -\mathcal{Y}_0^{h,\check\gamma}-\mathcal{Y}_0^\gamma
        \big)^2
      \Big]
    \Big)^\frac12\\
&   \leq
    \Big(
      \big|
        \E[\mathcal{Y}_0^\gamma-\mathcal{Y}_0^{h,\check\gamma}\big]
      \big|
      +
      \big(
        \E
        \big[
          \big(\mathcal{Y}_0^{h,\check\gamma}-\mathcal{Y}_0^\gamma\big)^2
        \big]
      \big)^\frac12
    \Big)
    \Big(
      \big|
        \E[\mathcal{Y}_0^\gamma+\mathcal{Y}_0^{h,\check\gamma}\big]
      \big|
      +
      \big(
        \E
        \big[
          \big(\mathcal{Y}_0^{h,\check\gamma}+\mathcal{Y}_0^\gamma\big)^2
        \big]
      \big)^\frac12
    \Big)\\
&   \leq
    4
    \Big(
      \big(
        \E\big[
          \big(
            \mathcal{Y}_0^\zeta
          \big)^2
        \big]
      \big)^\frac12
      +
      \sup_{h\in(0,1)}
      \big(
        \E\big[
          \big(
            \mathcal{Y}_0^{h,\check\zeta}
          \big)^2
        \big]
      \big)^\frac12
    \Big)
      \big(
        \E
        \big[
          \big(\mathcal{Y}_0^{h,\check\gamma}-\mathcal{Y}_0^\gamma\big)^2
        \big]
      \big)^\frac12.
\end{align*}
Together with \eqref{eq:y0}, this completes the proof for the term $|\Phi_\lambda(\zeta)-\Phi_{\lambda,h}(\check{\zeta})|$.

For the second term, $|\Phi_\lambda(\hat{\zeta}_h)-\Phi_{\lambda,h}(\zeta_h)|$, we define the functions $a_h\colon[0,T]\times\R^d\to\R^d$, $h\in(0,1)$, by $a_h=\hat{c}_h$, where $c_h\colon\{0,\dots,N\}\times\R^d\to\R^d$ are given by  $c_{h,n}(x)=b(t_n,x,\zeta_{h,n}(x))$. By the assumptions on $b$ and $\zeta_h$, we have that $a_h$ satisfy the assumption of Lemma~\ref{lemma_strong_conv}. This implies that $\lim_{h\to0}\|X^{\zeta_h}-\hat{X}^{h,\zeta_h}\|_{\mathcal{S}^2(\R^d)}=0$ and $\|X^{\zeta_h}-\hat{X}^{h,\zeta_h}\|_{\mathcal{S}^2(\R^d)}\leq C h^\alpha$ for $\alpha>0$. This fact, together with the same calculations as for $I_1,I_2,I_4,I_5$, and noting that the analogous for $I_3$ is zero, completes the proof.
\end{proof}
\end{lemma}

\subsection{Time discretization error of the initial and terminal values}
\label{sec:Y0}

This and the next section contain our main results. The results are based on an assumption of regularity of the Markov maps $\zeta^*\in\mathcal Z$ and $\zeta_{\lambda,h}^*\in\mathcal{Z}_h$, $h\in(0,1)$, $\lambda\geq0$, of the FBSDE and its discretizations, respectively. The regularity has to be verified for specific problems and this, in particular for the discrete problem, is a non-trivial and also not a well studied problem, see \cite{buttazzo1982gamma,rockafellar2009variational,bonnans2019time}. For the continuous problem, it breaks down into existence and uniqueness of a solution of the FBSDE and regularity of the solution to the HJB equation. In this paper, we provide only the abstract error analysis. Our assumption is next stated.

\begin{assumption}\label{as:Y0}
There exists a unique minimizer $\zeta^*\in\mathcal Z$ of $\zeta\mapsto \Phi_1(\zeta)$, it belongs to $\mathcal{Z}^{\alpha,\beta}$ and satisfies $\mathrm{Var}(\mathcal{Y}_0^{\zeta^*})=0$. Moreover, for all $\lambda\geq0$, there exists a collection of minimizers $\zeta_{\lambda,h}^*\in\mathcal{Z}_h$, $h\in(0,1)$, of $\zeta\mapsto\Phi_{\lambda,h}(\zeta)$ that for each $h$ belongs to $\mathcal{Z}_h^\beta$ and satisfies
$
    \sup_{h\in(0,1)}
    \triple
      \hat{\zeta}_{\lambda,h}^*
    \triple_{0,\beta}
    <\infty.
$
\end{assumption}

Our first result concerns the convergence of the objective function. The proof relies on a split of the error into two error terms, one containing only $\zeta^*$ and one containing only $\zeta_{\lambda,h}^*$. This way, we avoid getting an error bound containing the error between the continuous and discrete Markov maps. The result is used throughout our proofs, except in Proposition~\ref{prop_strong}, where we have the error of the Markov map in the bound. 

\begin{theorem}\label{thm_Y0}
Suppose the setting of Subsection~\ref{sec:setting}, let Assumption~\ref{as:Y0} hold and $\lambda\geq0$. Then, $|\Phi^*-\Phi_{\lambda,h}^*|\to0$ as $h\to0$, and if $\alpha>0$, then there exists a constant $C$, independent of $h$, such that
\begin{equation*}
    |\Phi^*-\Phi_{\lambda,h}^*|\leq C h^\alpha.
\end{equation*}
\begin{proof}
We start by noting that, since $\mathrm{Var}(\mathcal{Y}_0^{\zeta^*})=0$, it holds that $\zeta^*$ is optimal for all $\Phi_\lambda$, $\lambda\geq0$, and $\Phi^*=\Phi_\lambda^*:=\Phi_\lambda(\zeta^*)$. Because of optimality of $\zeta^*\in\mathcal Z$ and since $\hat{\zeta}_h^*\in\mathcal Z$, it holds that $\Phi^*=\Phi_\lambda^*\leq \Phi_\lambda(\hat{\zeta}_h^*)$. This implies
\begin{equation*}
    \Phi^*\leq \Phi_\lambda(\hat{\zeta}_{\lambda,h}^*)
    =
    \Phi_{\lambda,h}^* + \Phi_\lambda(\hat{\zeta}_{\lambda,h}^*) -\Phi_{\lambda,h}^*
    \leq 
    \Phi_{\lambda,h}^* +|\Phi_\lambda(\hat{\zeta}_{\lambda,h}^*) -\Phi_{\lambda,h}^*|.
\end{equation*}
Similarly, from the optimality of $\zeta_{\lambda,h}^*\in\mathcal{Z}_h$ and since $\check{\zeta}^*\in\mathcal{Z}_h$, it holds  $\Phi_{\lambda,h}^*\leq \Phi_{\lambda,h}(\check{\zeta}^*)$. We have
\begin{equation*}
   \Phi_{\lambda,h}^*\leq \Phi_{\lambda,h}(\check{\zeta}^*)
   =
   \Phi^*+\Phi_{\lambda,h}(\check{\zeta}^*)-\Phi^*
   \leq
   \Phi^*+|\Phi_{\lambda,h}(\check{\zeta}^*)-\Phi^*|.
\end{equation*}
The two inequalities equivalently read
$
    \Phi^*-\Phi_{\lambda,h}^*\leq|\Phi(\hat{\zeta}_{\lambda,h}^*) -\Phi_{\lambda,h}^*|$
and
$
    \Phi_{\lambda,h}^*-\Phi^*\leq |\Phi_{\lambda,h}(\check{\zeta}^*)-\Phi^*|,
$
and thus
\begin{equation*}
    |\Phi^*-\Phi_{\lambda,h}^*|\leq \max(|\Phi(\hat{\zeta}_{\lambda,h}^*) -\Phi_{\lambda,h}^*|,|\Phi_{\lambda,h}(\check{\zeta}^*)-\Phi^*|)\leq|\Phi(\hat{\zeta}_{\lambda,h}^*) -\Phi_{\lambda,h}^*|+ |\Phi_{\lambda,h}(\check{\zeta}^*)-\Phi^*|.
\end{equation*}
The convergence is given by our assumption and Lemma \ref{weak_conv_value_fun}. This completes the proof. 
\end{proof}
\end{theorem}

A consequence of the convergence of the objective function is the convergence of the two components of the objective, given that $\lambda>0$. This is stated in our next theorem.

\begin{theorem}\label{thm_Y02}
Suppose the setting of Subsection~\ref{sec:setting}, let Assumption~\ref{as:Y0} hold, $\lambda>0$, $\mathcal{Y}_0^{h,\lambda}\coloneqq \mathcal{Y}_0^{h,\zeta^*_{h,\lambda}}$ and $Y_0^{h,\lambda}\coloneqq \E[\mathcal{Y}_0^{h,\lambda}]$. Then, $|Y_0-Y_0^{h,\lambda}|+\mathrm{Var}(\mathcal{Y}_0^{h,\lambda})\to0$ as $h\to0$, and if $\alpha>0$, there exists a constant $C$, independent of $h$, such that
\begin{equation*}
    \big|
      Y_0-Y_0^{h,\lambda}
    \big|
    +
    \mathrm{Var}
    \big(\mathcal{Y}_0^{h,\lambda}\big)
    \leq C h^\alpha.
\end{equation*}
\begin{proof}
We start with the proof for the variance and use both the sequences $(\zeta_{0,h}^*)_{h\in(0,1)}$ and $(\zeta_{\lambda,h})_{h\in(0,1)}$. The proof relies on a squeezing argument. Adding and subtracting terms and using the triangle inequality, yields
\begin{align}\label{eq:var1}
    \lambda
    \mathrm{Var}(\mathcal{Y}_0^{h,\lambda})
    \leq
    \big|
      Y_0
      -
      Y_0^{h,\lambda}
      -
      \lambda\mathrm{Var}\big(\mathcal{Y}_0^{h,\lambda}\big)
    \big|
    +
    \big|
      Y_0-Y_0^{h,0}
    \big|
    +
    \big|
      Y_0^{h,0}-Y_0^{h,\lambda}
    \big|.
\end{align}
From the assumption $\mathrm{Var}(\mathcal{Y}_0)=0$, it holds that $\Phi^*=Y_0$ and Theorem~\ref{thm_Y0} gives
\begin{align}\label{eq:var2}
    \big|
      Y_0
      -
      Y_0^{h,\lambda}
      -
      \lambda\mathrm{Var}\big(\mathcal{Y}_0^{h,\lambda}\big)
    \big|
    +
    \big|
      Y_0-Y_0^{h,0}
    \big|
    =
    \big|
      \Phi^*
      -
      \Phi_{\lambda,h}^*
    \big|
    +
    \big|
      \Phi^*-\Phi_{0,h}^*
    \big|
    \leq C h^\alpha.
\end{align}

For the third term on the right-hand side of \eqref{eq:var1}, we first notice that, since $\zeta_{0,h}$ is a minimizer of $\zeta\mapsto Y_0^{h,0}$, it holds that
$      
      Y_0^{h,\lambda}
      -
      Y_0^{h,0}
      \geq0
$. This fact, adding $\lambda\mathrm{Var}(\mathcal{Y}_0^{h,\lambda})$, adding and subtracting $Y_0$, using the triangle inequality and using \eqref{eq:var2}, gives us
\begin{equation}\label{eq:var3}
\begin{split}
    \big|
      Y_0^{h,0}
      -
      Y_0^{h,\lambda}
    \big|
    &=
      Y_0^{h,\lambda}
      -
      Y_0^{h,0}
      \leq
      Y_0^{h,\lambda}
      +
      \lambda\mathrm{Var}
      \big(\mathcal{Y}_0^{h,\lambda}\big)
      -
      Y_0
      -
      Y_0^{h,0}
      +
      Y_0\\
    &\leq
    \big|
      Y_0
      -
      Y_0^{h,\lambda}
      -
      \lambda\mathrm{Var}
      \big(\mathcal{Y}_0^{h,\lambda}\big)
    \big|
    +
    \big|
      Y_0
      -
      Y_0^{h,\lambda}
    \big|
    \leq C h^\alpha.
\end{split}
\end{equation}
Now, \eqref{eq:var1}--\eqref{eq:var3} complete the proof for the variance. For the convergence of $Y_0^{h,\lambda}$, we conclude
\begin{align*}
  \big|Y_0-Y_0^{h,\lambda}\big|
  \leq
  \big|Y_0-Y_0^{h,\lambda}-\lambda\mathrm{Var}
    \big(\mathcal{Y}_0^{h,\lambda}\big)
  \big|
    +
    \lambda\mathrm{Var}
    \big(\mathcal{Y}_0^{h,\lambda}\big)
  \leq
  C h^\alpha.
\end{align*}
\end{proof}
\end{theorem}

From \eqref{eq:Var_identity} and Theorem~\ref{thm_Y02}, we directly get convergence in the terminal value and we can also conclude strong convergence of $y_0^{h,\lambda}$. This is stated in the following two corollaries.
\begin{corollary}\label{thm_Y03}
Suppose the setting of Subsection~\ref{sec:setting}, let Assumption~\ref{as:Y0} hold and $\lambda>0$. Then, $\E[
      (
      g(X_N^{h,\lambda})
      -
      Y_N^{h,\lambda}
      )^2
    ]\to0$ as $h\to0$, and if $\alpha>0$, then there exists a constant $C$, independent of $h$, such that
\begin{equation*}
    \Big(
    \E\Big[
      \big(
      g\big(X_N^{h,\lambda}\big)
      -
      Y_N^{h,\lambda}
      \big)^2
    \Big]
    \Big)^\frac12
    \leq C h^\frac\alpha2.
\end{equation*}
\end{corollary}
\begin{corollary}\label{thm_Y04}
Suppose the setting of Subsection~\ref{sec:setting}, let Assumption~\ref{as:Y0} hold and $\lambda>0$. Then, $(\E[
      (
      \mathcal{Y}_0
      -
      \mathcal{Y}_0^{h,\lambda}
      )^2
    ])^\frac12\to0$ as $h\to0$, and if $\alpha>0$, then there exists a constant $C$, independent of $h$, such that
\begin{equation*}
    \Big(
    \E\Big[
      \big(
      \mathcal{Y}_0
      -
      \mathcal{Y}_0^{h,\lambda}
      \big)^2
    \Big]
    \Big)^\frac12
    \leq C h^\frac\alpha2.
\end{equation*}
\begin{proof}
From Theorem~\ref{thm_Y02} and the triangle inequality, it holds
\begin{align*}
\Big(
    \E\Big[
      \big(
      \mathcal{Y}_0
      -
      \mathcal{Y}_0^{h,\lambda}
      \big)^2
    \Big]
    \Big)^\frac12
    &\leq
    \Big(
    \E\Big[
      \big(
      Y_0-Y_0^{h,\lambda}
      +      
      \mathcal{Y}_0
      -
      \mathcal{Y}_0^{h,\lambda}
      \big)^2
    \Big]
    \Big)^\frac12
    +
    \big|
      Y_0-Y_0^{h,\lambda}
    \big|\\
    &=
    \Big(
      \mathrm{Var}
      \big(\mathcal{Y}_0^{h,\lambda}\big)
    \Big)^\frac12
    +
    \big|
      Y_0-Y_0^{h,\lambda}
    \big|
    \leq C h^\frac\alpha2.
\end{align*}
This proves the corollary.
\end{proof}
\end{corollary}

\subsection{Time discretization error of the FBSDE} \label{sec:strong}
While the error analysis for the initial and terminal values in the previous subsection required only Assumption~\ref{as:Y0}, the strong error analysis of $(X,Y,Z)$ requires more. In Theorem~\ref{thm_strong}, we prove strong convergence for small horizon $T$. This should be compared with the convergence result in \cite{han2020convergence}, that also has a very restrictive assumption on $T$. Compared to \cite{han2020convergence}, whose bound contains the error in the terminal value, we have no such term. In Proposition~\ref{prop_strong}, we prove the strong convergence without a restriction on $T$ with the cost of having a bound containing the error between the Markov maps $\zeta^*$ and $\zeta_{\lambda,h}^*$.

\begin{theorem}\label{thm_strong}
Suppose the setting of Subsection~\ref{sec:setting}, let Assumption~\ref{as:Y0} hold, $\alpha,\lambda>0$, and 
\begin{align*}
  \max\Bigg(
    T^\frac12|f|_{0,0,1}, 
    \frac{5T(T|f|_{0,1,0}+|g|_1)
    |b|_{0,0,1}^2
    \exp(
      5T
      (
        |b|_{0,1,0}T+|\sigma|_{0,1}
      )      
    )}{1-T^\frac12|f|_{0,0,1}}
    \Bigg)<1.
\end{align*}
Then, there exists a constant $C$, independent of $h$, such that
\begin{align*}
    \big\|
      X-\hat{X}^{h,\lambda}
    \big\|_{\mathcal{S}^2(\R^d)}
    +
    \big\|
      Y-\hat{Y}^{h,\lambda}
    \big\|_{\mathcal{S}^2(\R^d)}
    +
    \big\|
      Z-\hat{Z}^{h,\lambda}
    \big\|_{\mathcal{H}^2(\R^k)}
    \leq
    C h^{\frac\alpha2}.
\end{align*}

\begin{proof}
We start by noting that for $t\in[0,T]$, it holds
\begin{align}\label{eq:BSDEt0}
    Y_t - Y_t^{h,\lambda}
    =
    Y_0 - Y_0^{h,\lambda}
    -
    \int_0^t 
    \big(
      f(s,X_s,Z_s)
      -
      f(\bar s,\hat{X}^{h,\lambda}_s,\hat{Z}_s^{h,\lambda})
    \big)
    \d s
    +
    \int_0^t
      \big(
        Z_s-\hat{Z}_s^{h,\lambda}
      \big)
    \d W_s.
\end{align}
By the It\^o Isometry, substitution of \eqref{eq:BSDEt0} with $t=T$, and the triangle inequalities, we have
\begin{align*}
    &\big\|
      Z-\hat{Z}^{h,\lambda}
    \big\|_{\mathcal{H}^2(\R^k)}
    =
    \Big(\E
    \Big[
      \int_0^T
        \|Z_t-\hat{Z}_t^{h,\lambda}\|^2
      \d t
    \Big]
    \Big)^\frac12
    =
    \Big(
    \E
    \Big[
      \Big(
      \int_0^T
        (Z_t-\hat{Z}_t^{h,\lambda})
      \d W_t
      \Big)^2
    \Big]
    \Big)^\frac12\\
    &\quad
    =
    \Bigg(
    \E
    \Bigg[
      \Bigg(
      Y_0 - Y_0^{h,\lambda}
      +
      g(X_T)-Y^{h,\lambda}_N
      -
      \int_0^T 
      \big(
        f(t,X_t,Z_t)
        -
        f(\bar t,\hat{X}^{h,\lambda}_t,\hat{Z}_t^{h,\lambda})
      \big)
      \d t
    \Bigg)^2
    \Bigg]
    \Bigg)^\frac12\\
        &\quad
    \leq
    \big|
      Y_0-Y_0^{h,\lambda}
    \big|
    +
      \big(
        \E\big[
          \big(
            g(X_T)-g(X^{h,\lambda}_N)
          \big)^2
        \big]
      \big)^\frac12
    +  
      \big(
        \E\big[
          \big(
            g(X^{h,\lambda}_N)-Y_N^{h,\lambda}
          \big)^2
        \big]
      \big)^\frac12\\      
    &\qquad
    +
    \Big(
      \E
      \Big[
      \Big(
      \int_0^T 
      \big(
        f(t,X_t,Z_t)
        -
        f(\bar t,\hat{X}^{h,\lambda}_t,\hat{Z}_t^{h,\lambda})
      \big)
      \d t
      \Big)^2
      \Big]
      \Big)^\frac12.
\end{align*}
The first three terms are by Theorem~\ref{thm_Y02} and Corollary~\ref{thm_Y03}, bounded from below by $Ch^{\frac\alpha2}+|g|_1\|X-\hat{X}^{h,\lambda}\|_{\mathcal{S}^2(\R^d)}$. By similar arguments as those for $I_2,I_3$ in the proof of Lemma~\ref{weak_conv_value_fun}, it holds
\begin{equation}\label{eq:ff}
\begin{split}
    &\Big(
      \E
      \Big[
      \Big(
      \int_0^T 
      \big(
        f( t,X_t,Z_t)
        -
        f(\bar t,\hat{X}^{h,\lambda}_t,\hat{Z}_t^{h,\lambda})
      \big)
      \d t
      \Big)^2
      \Big]
    \Big)^\frac12\\
    &\quad
    \leq
    T|f|_{\alpha,0,0}(1+\alpha)^{-1}h^\alpha+T|f|_{0,1,0}\|X-\hat{X}^{h,\lambda}\|_{\mathcal{S}^2(\R^d)}
    +
    T^\frac12
    |f|_{0,0,1}
    \big\|
      \hat{Z}^{h,\lambda}-Z
    \big\|_{\mathcal{H}^2(\R^k)}.
\end{split}
\end{equation}
By a kickback argument and by assumption, it holds
\begin{align}\label{eq:ZX}
    \big\|
      Z-\hat{Z}^{h,\lambda}
    \big\|_{\mathcal{H}^2(\R^k)}
    \leq
    C h^\alpha+\frac{T|f|_{0,1,0}+|g|_1}{1-T^\frac12|f|_{0,0,1}}\|X-\hat{X}^{h,\lambda}\|_{\mathcal{S}^2(\R^d)}.
\end{align}

We next approach the error $\|X-\hat{X}^{h,\lambda}\|_{\mathcal{S}^2(\R^d)}$. Let $\Xi^{h,\lambda}\in\mathcal{S}^2(\R^d)$, $h\in(0,1)$, be the family of stochastic processes that for all $t\in[0,T]$, $\P$-a.s., satisfy
\begin{align*}
    \Xi_t^{h,\lambda}
    =
    x_0
    +
    \int_0^t
      b(\bar s,\Xi^{h,\lambda}_s,Z_s^{h,\lambda})
    \d s + \int_0^t\sigma(\bar s,\Xi_s^{h,\lambda})\d W_s.
\end{align*}
Using the triangle inequality, we get
\begin{align*}
    &\|X-X^{h,\lambda}\|_{\mathcal{S}^2(\R^d)}
    \leq
    \|X-\Xi^{h,\lambda}\|_{\mathcal{S}^2(\R^d)}
    +
    \|\Xi^{h,\lambda}-\hat{X}^{h,\lambda}\|_{\mathcal{S}^2(\R^d)}.
\end{align*}
By the arguments used repeatedly in the proof of Lemma~\ref{lemma_strong_conv} and by assumptions, it holds
\begin{align*}
    \E\big[\|X_t-&\Xi^{h,\lambda}_t\|^2\big]\\
    &\leq
    C
    h^\alpha
    +
    5|b|_{0,0,1}^2T
    \|Z-Z^{h,\lambda}\|^2_{\mathcal{H}^2(\R^k)}
    +
    5\big(
      |b|_{0,1,0}T+|\sigma|_{0,1}
    \big)
    \int_0^t
      \E\big[\|X_s-\Xi^{h,\lambda}_s\|^2\big]
    \d s.
\end{align*}
From this, it follows by the Gronwall lemma that 
\begin{align*}
    \E\big[\|X_t-\Xi^{h,\lambda}_t\|^2\big]
    \leq
    \exp
    \Big(
      5T
      \big(
        |b|_{0,1,0}T+|\sigma|_{0,1}
      \big)      
    \Big)
    \Big(
      C
      h^\alpha
      +
      5|b|_{0,0,1}^2T
      \|Z-Z^{h,\lambda}\|^2_{\mathcal{H}^2(\R^k)}
    \Big).
\end{align*}
A use of Lemma~\ref{lemma_strong_conv} yields $\|\Xi^{h,\lambda}-\hat{X}^{h,\lambda}\|_{\mathcal{S}^2(\R^d)}\leq Ch^\alpha$. We conclude that
\begin{align}\label{eq:XZ}
    &\|X-\hat{X}^{h,\lambda}\|_{\mathcal{S}^2(\R^d)}
    \leq
    Ch^\alpha
    +
    5T
    |b|_{0,0,1}^2
    \exp(
      5T
      (
        |b|_{0,1,0}T+|\sigma|_{0,1}
      )      
    )
    \|Z-\hat{Z}^{h,\lambda}\|_{\mathcal{H}^2(\R^k)}.
\end{align}
Using \eqref{eq:XZ} in \eqref{eq:ZX} gives, after a kickback argument, the desired bound 
$\|Z-\hat{Z}^{h,\lambda}\|_{\mathcal{H}^2(\R^k)}\leq Ch^\alpha$.

For $\|Y-\hat{Y}^{h,\lambda}\|_{\mathcal{S}^2(\R^d)}$, we use \eqref{eq:BSDEt0} and the standard arguments to get
\begin{align*}
    \Big(\E
    \big[\big\|
      Y_t&-Y_t^{h,\lambda}
    \big\|^2\big]
    \Big)^\frac12\\
    &\leq
    \big|Y_0-Y_0^{h,\lambda}\big|
    +
    \Bigg(
    \int_0^T
      \E\Big[
        \big\|
        f(s,X_s,Z_s)
        -
        f(\bar s, \hat{X}_s^{h,\lambda},\hat{Z}_s^{h,\lambda})
        \big\|^2
      \Big]
    \d s
    \Bigg)^\frac12
    +
    \big\|
        Z_s-\hat{Z}_s^{h,\lambda}
    \big\|_{\mathcal{H}^2(\R^k)}.
\end{align*}
Applying Theorem~\ref{thm_Y0}, \eqref{eq:ff} and the obtained results for $\|X-\hat{X}^{h,\lambda}\|_{\mathcal{S}^2(\R^d)}$ and $\|Z-\hat{Z}^{h,\lambda}\|_{\mathcal{H}^2(\R^k)}$ complete the proof.
\end{proof}
\end{theorem}

\begin{proposition}\label{prop_strong}
Suppose the setting of Subsection~\ref{sec:setting}, let Assumption~\ref{as:Y0} hold, $\alpha,\lambda>0$. Then, there exists a constant $C$, independent of $h$, such that
\begin{align*}
    &\big\|
      X-\hat{X}^{h,\lambda}
    \big\|_{\mathcal{S}^2(\R^d)}
    +
    \big\|
      Y-\hat{Y}^{h,\lambda}
    \big\|_{\mathcal{S}^2(\R^d)}
    +
    \big\|
      Z-\hat{Z}^{h,\lambda}
    \big\|_{\mathcal{H}^2(\R^k)}\\
    &\qquad
    \leq C\Bigg(h^{\frac\alpha2}+
    \max_{0,\dots,N_h}
    \Big(
    \E
    \Big[
    \big\|
      \zeta^*(t_n,X_n^{\lambda,h})-\hat{\zeta}_{h,\lambda}^*(t_n,X_n^{\lambda,h})
    \big\|^2
    \Big]
    \Big)^\frac12
  \Bigg).
\end{align*}
\begin{proof}
This is proved similarly to Theorem~\ref{thm_strong} without the kick-back argument and instead of \eqref{eq:ff}, using 
\begin{align*}
    &\Big(
      \E
      \Big[
      \Big(
      \int_0^T 
      \big(
        f( t,X_t,Z_t)
        -
        f(\bar t,\hat{X}^{h}_t,\hat{Z}_t^{h})
      \big)
      \d t
      \Big)^2
      \Big]
    \Big)^\frac12\\
    &\quad
    \leq
    C\Big(
      h^{\frac\alpha2}
      +
      \|X-\hat{X}^{h,\lambda}\|_{\mathcal{S}^2(\R^d)}
      +
    \max_{0,\dots,N_h}
    \Big(
    \E
    \Big[
    \big\|
      \zeta^*(t_n,X_n^{\lambda,h})-\hat{\zeta}_{h,\lambda}^*(t_n,X_n^{\lambda,h})
    \big\|^2
    \Big]
    \Big)^\frac12
    \Big).
\end{align*}
\end{proof}
\end{proposition}

\subsection{A discussion on the full error analysis of the robust deep FBSDE method}\label{sec:full}

In Subsections~\ref{sec:Y0} and \ref{sec:strong}, only the time discretization error is considered, i.e., the error between \eqref{var_FBSDE} and \eqref{var_discrete_FBSDE}. For a full error analysis, the error between the fully implementable scheme \eqref{var_discrete_FBSDE_implementable} and \eqref{var_FBSDE} must be considered. Besides the time discretization error, there are three other sources of error: The first is the error induced by optimizing over the parameters of a neural network, instead of over the vast set $\mathcal{Z}_h$. By the Universal Approximation Theorem \cite{cybenko1989approximations}, this error can be made arbitrarily small, but this theorem gives no help with suggesting the network architecture that can guarantee a maximal error of desired size. The second error is the Monte-Carlo error induced from approximating the expectation in $Y_0$ by a sample mean. This error allows for a simple error analysis and the Monte Carlo error is of the order $\mathcal O(M_{\mathrm{batch}}^{-1/2}$). The final error is the error induced from the inexact optimization procedure of \eqref{var_discrete_FBSDE_implementable}. 

\section{Numerical experiments}\label{sec5}
In this section, we evaluate our algorithm on three different problems. The first two are of LQ type, for which we have access to a semi-analytic solution for comparison. The third example uses nonlinear terms, both in the drift and diffusion coefficients in the forward equation, and we no longer have access to a reference solution. In the first example, there is a one-to-one map between the feedback control and the $Z-$process,  and we can set $\lambda=0$ in the loss function. In the second and third examples, this is not the case, and $\lambda>0$ is necessary for uniqueness of the minimizer to our discrete problem, and in turn convergence to the continuous FBSDE.

In the experimental convergence studies, we approximate $\|\cdot\|_{\mathcal{S}(\R^q)}$ and $\|\cdot\|_{\mathcal{H}(\R^q)}$, with
\begin{equation*}
    \|A\|_{\mathcal{S}_{h,M}^2(\R^q)}=\max_{n\in\{0,1,\ldots,N\}}\bigg(\frac{1}{M}\sum_{m=1}^M\|A_n(m)\|^2\bigg)^{\frac12}, \quad \|A\|_{\mathcal{H}_{h,M}^2(\R^q)}=\frac{1}{N}\sum_{N=0}^{N-1}\bigg(\frac{1}{M}\sum_{m=1}^M\|A_n(m)\|^2\bigg)^{\frac12}.
\end{equation*}
Here, $A(m)=\{A_1(m),A_2(m),\cdot,A_N(m)\}$, $m=1,2,\ldots,M$, are \textit{i.i.d.} realizations of some adapted stochastic processes $A$ on the grid. The norm $\|\cdot\|_{L^2(\Omega;\R^q)}$
is approximated with a sample mean, denoted $\|\cdot\|_{L_{h,M}^2(\R^q)}$. For the convergence study, the Experimental Order of Convergence (EOC) is used. It is defined as
 \begin{equation*}
     \text{EOC}(h_{i})=\frac{\log{(\text{error}(h_{i+1}))} - \log{(\text{error}(h_i))}}{\log{(h_{i+1})}-\log{(h_i)}}.
 \end{equation*}
In all examples, we use the neural network architecture in Section~\ref{NN_spec}. We use $M_\text{train}=2^{22}$ training data points and batch size $M_\text{batch}=2^9$ with $K_{\mathrm{epoch}}=15$ epochs. This gives $K_{\mathrm{batch}}=2^{12}=4096$ updates per epoch.
For the optimization, the Adam optimizer \cite{kingma2014adam} is used with learning rate 0.1 for the first three epochs, which, after that, is multiplied by a factor of $\e^{-0.5}$ for each new epoch. For our use, it was important to choose $M_{\mathrm{train}}$ large, since in our empirical convergence results we want to  isolate the time discretization error. In practice, the method generates acceptable solutions with significantly smaller $M_\text{train}$.

\subsection{Linear quadratic control problems}\label{lqgcp_num}
Among all stochastic control problems, the LQ control problem is the most studied and that with the most structure, see \textit{e.g.,} \cite{aastrom2012introduction}. For our purposes,  it has a closed-form analytic solution, with which we can compare our numerical approximations. 

Let $k=d$, $x_0\in\R^d$, $A,\sigma,\in\R^{d\times d}$, $R_x,G\in\mathbbm{S}^d_+$,  $R_u\in\mathbbm{S}^\ell_+$ and $B\in\R^{d\times \ell}$ be of full rank and $C\in\R^d$. The state equation and cost functional of a linear-quadratic-Gaussian control problem are given by
\begin{equation*}\label{LQR_md}
   \begin{dcases}  X_t=x_0+\int_0^t\big(A(C-X_s) + Bu_s\big)\d s + \int_0^t\sigma\d W_s, \\
    J^u(t,x)=\E^{t,x}\Big[\int_t^T (\langle R_x X_s, X_s\rangle + \langle R_u u_s,u_s\rangle)\d s+\langle G X_T, X_T\rangle\Big],\quad t\in[0,T].
    \end{dcases}
\end{equation*}
With the minimizer $v^*$ of the corresponding Hamiltonian, $\inf_{u\in U}\{\langle\text{D}_x V, Bu\rangle+ \langle R_u u, u\rangle\}$, we have the optimal feedback control
\begin{equation}\label{opt_control}
    u^*_t=-\frac{1}{2}R_u^{-1}B^T\text{D}_x V(t,X_t).
\end{equation}
Here, we recall that $V$ is the solution to the associated HJB-equation. Its solution is given by
\begin{align*}
    V(t,x)
    =
    x^TP(t)x
    +
    x^TQ(t)
    +
    R(t),
\end{align*}
where $(P,Q,R)$ solves the system of ordinary differential equations,
\begin{equation*}\begin{dcases}
     \dot{P}(t)-A^TP(t)-P(t)A-P(t)BR_u^{-1}B^TP(t)+R_x=\mathbf{0}_{d\times d},\\
    \dot{Q}(t) +2P(t)AC-A^TQ(t)-P(t)BR_u^{-1}B^TQ(t)=\mathbf{0}_d,\\
    \dot{R}(t)+\text{Tr}\big\{\sigma\sigma^TP(t)\big\} + Q(t)^TAC-\frac{1}{4}Q(t)^TBR_u^{-1}B^TQ(t)=0,\quad t\in[0,T],\\
    P(T)=G;\quad Q(T)=\mathbf{0}_d; \quad R(T)=0.
    \end{dcases}
\end{equation*}
The first equation is a matrix Riccati equation, and we refer to the whole system, slightly inaccurately, as the Riccati equation. The gradient of $V$ satisfies $\text{D}_xV(t,x)=2P(t)x+Q(t)$. The related FBSDE reads:
\begin{equation}
    \begin{dcases}
    X_t =x_0+\int_0^t \big[A(C-X_s)-\frac{1}{2}BR_u^{-1}B^\top Z_s\big]\d s + \int_0^t\sigma\d W_s,\\
    Y_t =\langle GX_T,X_T\rangle -\int_t^T\big(\langle R_x X_s,X_s\rangle-\frac{1}{4}\langle R_u^{-1}B^\top Z_s, B^\top Z_s\rangle\big)\d s + \int_t^T\langle Z_s,\sigma \d W_s\rangle,\quad t\in[0,T].
    \end{dcases}\label{FBSDE_LQG}
\end{equation}
The solution to \eqref{FBSDE_LQG} is then given by \begin{equation}\label{eq:YZ}
    Y_t=X_t^TP(t)X_t+X_t^TQ(t) +R(t);\quad Z_t=2P(t)X_t + Q(t).
\end{equation}
The Riccati equation has an analytic solution in closed-form, only in one dimension. As benchmark solution in our experiments, we use the Euler approximation of the Riccati equation with $160\times 2^7$ time steps and with $160$ time steps for for $X$. The processes $(Y,Z)$ are approximated by \eqref{eq:YZ}.

\subsubsection{Example with state and control of the same dimension}\label{del}

Our first example concerns a two-dimensional LQ control problem with two-dimensional control. The matrices for the forward equation are given by
\begin{equation*}
    A=\begin{pmatrix}1&0\\0&2\end{pmatrix},\quad B=\begin{pmatrix}1&0.5\\-0.5&1\end{pmatrix},\quad
    C=\begin{pmatrix}0.1\\0.2\end{pmatrix},\quad
    \sigma=\begin{pmatrix}0.05&0.25\\ 0.05&
    0.25\end{pmatrix},\quad
    x_0=\begin{pmatrix}0.1\\0.1\end{pmatrix},\quad T=0.5,
\end{equation*}
and the penalty matrices for the control problem by
\begin{equation*}
    R_x=\begin{pmatrix}100&0\\0&1\end{pmatrix},\quad R_u=\begin{pmatrix}1&0\\0&1\end{pmatrix},\quad
    G=\begin{pmatrix}1&0\\0&100\end{pmatrix}.
\end{equation*}

In Figure \ref{XYZ}, the approximation of $(X,Y,Z)$ is compared to the analytic solution in mean, an empirical credible interval (again, defined as the area between the $5$:th and $95$:th percentiles at each time point) as well as for a single path. We see that the largest error comes from the approximation of $Y$. The reason for this is the error accumulation stemming from our time discretization. It is not due to the neural network approximation. This can be verified by using the baseline for $Z$, from the Riccati equation, and use an Euler-Maruyama scheme to generate the same error. This suggests that a more suitable choice of numerical schemes for $Y$ should be used.   
 
In Table \ref{table1}, we see the convergence rates from the experiment. The regime of the LQ control problem, with, e.g., quadratic dependence in $f$ does not satisfy the assumptions made in Section~\ref{sec4} and a direct comparison cannot be made. Still, we see, for instance, that $Y_0^h$ converges empirically with order $1$, while the error in the terminal condition reaches $0.69$ and is likely to continue to decrease. In Theorem~\ref{thm_Y02} and Corollary~\ref{thm_Y03}, there is a difference of a factor two between these two errors, which roughly appears to be in line with the rates obtained. 
 
\begin{figure}[htp]
\centering
\begin{tabular}{ccc}
          \includegraphics[width=53mm]{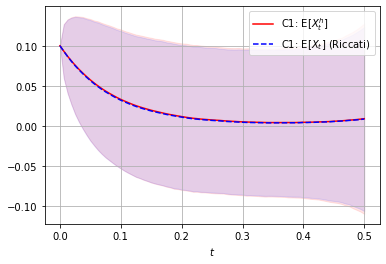}& \includegraphics[width=53mm]{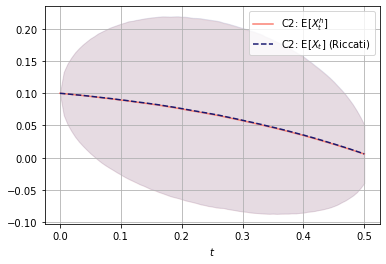}&  \includegraphics[width=53mm]{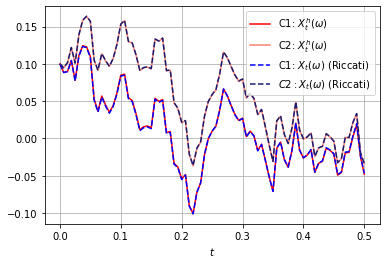}\\  \includegraphics[width=53mm]{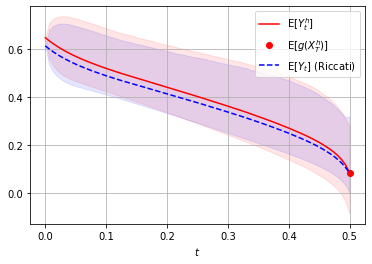}&&
          \includegraphics[width=53mm]{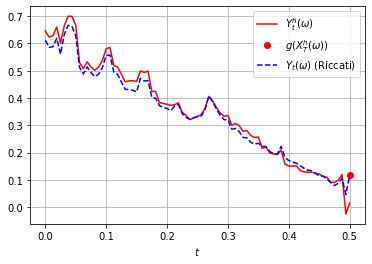} \\
          \includegraphics[width=53mm]{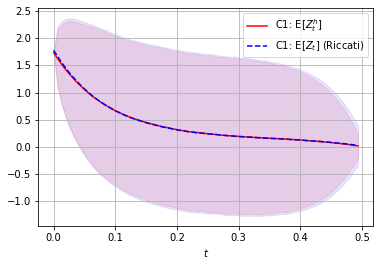}&  \includegraphics[width=53mm]{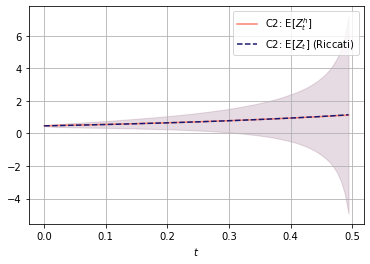}&
           \includegraphics[width=53mm]{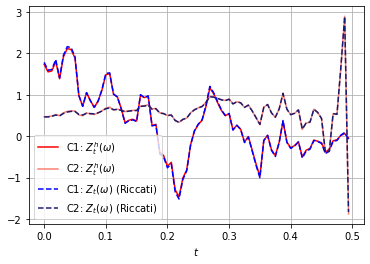}
\end{tabular}
\caption{Average of solutions and a single solution path compared to their analytic counterparts for the LQ control problem from Section~\ref{del}. The shaded areas represent empirical credible intervals, defined as the areas between the 5:th and the 95:th percentiles at each time point.}\label{XYZ}
\end{figure}

\begin{table}[]
\resizebox{\columnwidth}{!}{%
\begin{tabular}{llllllllllll}
\textbf{} & \multicolumn{2}{l}{ $\|X- X^h\|_{\mathcal{S}_{h}^2}$} & \multicolumn{2}{l}{ $\|Y- Y^h\|_{\mathcal{S}_h^2}$} & \multicolumn{2}{l}{ $\|Z- Z^h\|_{\mathcal{H}_h^2}$} & \multicolumn{2}{l}{$\|Y_N^h-g(X^h_N)\|_{L_h^2}$} & \multicolumn{2}{l}{ $|Y_0- Y_0^h|$} & $Y_0^h$ \\ \hline
         $N$ & Error & EOC & Error & EOC & Error & EOC & Error & EOC & Error & EOC & Value  \\ \hline
          \multicolumn{12}{c}{Problem 1 with $d=2$ and $\ell=2$ with analytic initial value $Y_0=0.6122$}.
         \\\hline 
         5 & 6.90e-2  & 1.28 & 1.22    & 1.03 & 5.66e-1 & 1.34 & 9.91e-1 &  0.98&7.10e-1& 1.13& 1.32 \\
         10 & 2.85e-2 & 1.11 & 6.02e-1 & 0.98 & 2.22e-1 & 1.16 & 5.04e-1 & 0.92 & 3.26e-1 & 1.16 & 0.937\\
         20 & 1.32e-2 & 1.09 & 3.05e-1 & 0.86 & 1.00e-1 & 1.08 & 2.67e-1 & 0.80 &1.47e-1& 1.04 & 0.759\\
         40 & 6.16e-3 & 1.00 & 1.68e-1 & 0.74 & 4.72e-2 & 0.98 & 1.53e-1 & 0.69 &7.01e-2& 1.01 & 0.683\\
         80 & 3.07e-3 &      & 1.01e-1 &      & 2.39e-2 &      & 9.46e-3 &      & 3.47e-2 &    & 0.645\\ \hline
        \multicolumn{12}{c}{Problem 2 with $d=6$ and $\ell=2$ with analytic initial value $Y_0=1.4599$.}
\\ \hline
5& 7.25e-2 & 1.23   & 5.65e-1 & 0.90  & 1.20    & 0.90  & 4.43e-1 & 0.80  & 3.51e-1 & 1.10  & 1.80\\ 
10& 3.10e-2 & 1.10 & 3.02e-1 & 0.76 & 7.41e-1 & 0.49 & 2.54e-1 & 0.69 & 1.63e-1 & 0.94 & 1.55\\ 

20& 1.45e-2 & 0.87 & 1.79e-1 & 0.63 & 5.26e-1 & 0.28 & 1.57e-1 & 0.60 & 8.51e-2 & 0.82 & 1.55\\ 
40& 7.96e-3 & 0.35 & 1.15e-1  & 0.53 & 4.34e-1 & 0.18 & 1.04e-1 & 0.54 & 4.81e-2 & 0.69 & 1.51\\ 
80& 6.24e-3 &      & 7.96e-2 &      & 3.84e-1 &      & 7.15e-2 &      & 3.00e-2 &      & 1.49\\ \hline
        \multicolumn{12}{c}{Problem 3 with $d=25$ and $\ell=1$ with analytic initial value $Y_0=11.348$.}
\\ \hline
5& 2.25e-1 & 1.86 & 1.93 & 0.68 & 3.40 & - & 1.29 & 0.50 & 1.43 & 0.99 & 12.78 \\ 
10& 6.19e-2 & 0.97 & 1.21 & 0.42 & 2.54 & - & 9.15e-1 & 0.31 & 0.72 & 0.47 & 12.07\\ 20& 5.66e-2 & 0.57 & 9.00e-1 & 0.29 & 2.73 & - & 7.40e-1 & 0.19 & 0.52 & 0.53 & 11.87\\ 
40& 2.48e-2 & 0.25 & 7.37e-1  & 0.069 & 2.72 & - & 6.47e-1 & 0.023 & 0.36 & 0.53 & 11.71\\ 
80& 2.09e-2 &      & 7.03e-1 &      & 3.06 &      & 6.37e-1 &      & 0.25 &      & 11.60\\ 

\end{tabular}}
\caption{Errors and experimental order of convergence for LQ control problems described in Sections \ref{del} and \ref{d6ell2}.}\label{table1}
\end{table}


\subsubsection{Example with control in lower dimensions than the state}\label{d6ell2}

Our second example concerns a six-dimensional problem with a two-dimensional control. The matrices used for the state equation are given by
 \begin{align*}
    A&=\text{diag}([1,2,3,1,2,3]),\quad B=\begin{pmatrix}1&-1\\1&1\\0.5&1\\1&-1\\0&-1\\0&1\end{pmatrix},\quad C=\text{diag}([-0.2,-0.1,0,0,0.1,0.2]),\\
    \sigma&=\text{diag}([0.05,0.25,0.05,0.25,0.05,0.25]),\quad x_0=(0.1,0.1,0.1,0.1,0.1,0.1)^\top,\quad T=0.5.
\end{align*}
The penalty matrices of the control problem are given by 
\begin{equation*}
        R_x=\text{diag}([25,1,25,1,25,1]),\quad R_u=\text{diag}([1,1]),\quad
    G=\text{diag}([1,25,1,25,1,25]).
\end{equation*}
Before we discuss our results, recall that the optimal feedback control at time $t$ is given by 
$
    u_t^* =-\frac{1}{2}R_u^{-1}B^TZ_t
$.
Since $u^*_t$ takes on values in $\R^\ell$ and $R_u^{-1}B^T$ is of rank $\ell<d$ at most, we can conclude that there exists infinitely many processes $\zeta_t$, such that $ u^*_t =-\frac{1}{2}R_u^{-1}B^T\zeta_t$. To obtain uniqueness of the control component, we set $\lambda=1>0$.


In Figure \ref{XYZd6ell2}, the approximations are compared with semi-analytic solutions in empirical mean, credible interval and for a representative path realization of $X,Y$ and $Z$. Visually, the approximations  capture $(X,Y,Z)$ well. The convergence is shown in the middle part of Table~\ref{table1} and we note that the experimental orders decrease below the orders of the previous example (top part of Table~\ref{table1}). To investigate whether this is the true convergence order, or if other errors are dominating for small time steps, we have done some hyperparameter optimization with different training data and batch sizes, learning rates and neural network architectures, without being able to improve these rates. 
\begin{figure}[htp]
\centering
\begin{tabular}{ccc}
          \includegraphics[width=80mm]{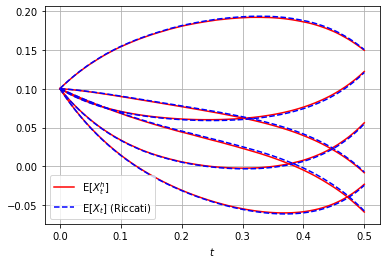}&  \includegraphics[width=80mm]{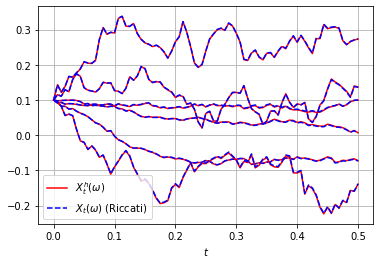}\\ \includegraphics[width=80mm]{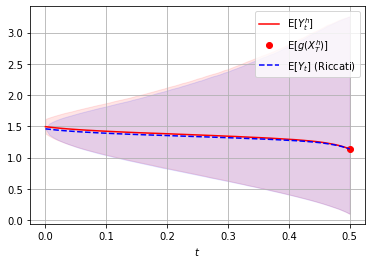}&
          \includegraphics[width=80mm]{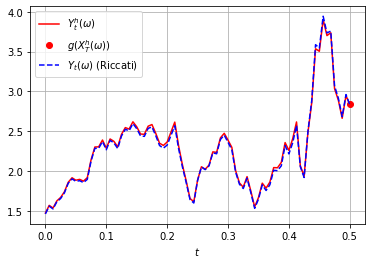}\\
          \includegraphics[width=80mm]{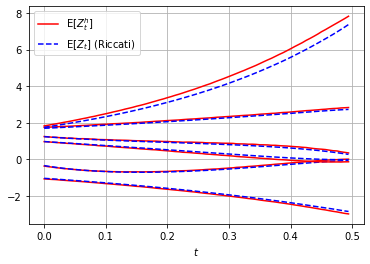}&   
         \includegraphics[width=80mm]{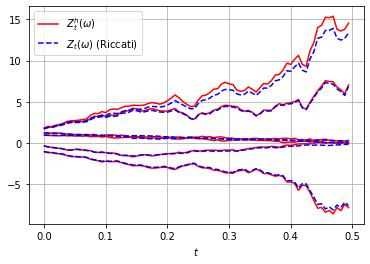}
\end{tabular}
\caption{Average of solutions and a single solution path compared to their analytic counterparts for the LQ control problem from Section~\ref{d6ell2}. The shaded area represents an empirical credible interval for $Y$, defined as the area between the 5:th and the 95:th percentiles at each time point. We do not include credible intervals for $X$ and $Z$ in this figure to facilitate visualization. For $X$ and $Z$, we see one realization of each of the six components.}\label{XYZd6ell2}
\end{figure}

The third example aims to demonstrate our methods' ability to deal with high-dimensional problems. Most high-dimensional PDE and BSDE problems in the literature are symmetric (solutions are permutation invariant), and in some cases the solutions can be represented by a one-dimensional BSDE \cite[Example 1]{kapllani2020deep}. From the parameters below, it becomes evident that the 25-dimensional problem that we choose is highly non-symmetric and therefore very challenging (arguably more challenging than a similar, but 100-dimensional symmetric problem). Non-symmetric problems in the literature are \cite{dai2021learning, ji2020three, liu2021deep, pereira2019learning, wang2022deep}, and the dimensions are 4, 5, 3, 4, and 2, respectively. A symmetric problem in 100-dimensions is found in \cite{ji2020three}. 

Despite the challenging nature of the problem and its relatively high-dimension, we achieve acceptable results, which is displayed in the bottom third of Table \ref{table1}. It should however be pointed out that the error source induced by the time discretization is no longer dominating. This means that we do not see a convergence with the number of time steps for the approximation of the $Z-$process. All the other discretization errors decreases with the step size, but it is clear that we have other significant error sources. Figure~\ref{XYZd25ell1} shows that visually the performance for 40 time steps is acceptable, even though some of the components of the $Z$ process oscillates close to the terminal time. The phenomena of accurate $X$ and $Y$ processes and less accuracy in some of the components of the $Z$ process could, at least heuristically, be explained by the mapping $\R^{25}\ni Z_t\mapsto u_t\in\R$. It is reasonable to assume that some of the components of the $Z$ process are more influential in the above mentioned mapping, which is what we have seen empirically in our experiments. Moreover, we have noticed that the components of the $Z$ process with the lowest magnitudes are less accurately approximated (relatively), which by the form of the feedback control, also justifies the above reasoning.

We use the following parameters:
$T=0.5$, $d=25$, $l=1$, $A=\text{diag}([1,2,3,\ldots,1,2,3,1])$,\\  $B=(1,1,0.5,1,0,0,1,1,0.5,1,0,0,1,1,0.5,1,0,0,1,1,0.5,1,0,0,1)$, $C=(-0.2,-0.1,0,0,0.1,0.2,-0.2,-0.1,0,\\0,0.1,0.2,-0.2,-0.1,0,0.,0.1,0.2,-0.2,-0.1,0,0,0.1,0.2,-0.2)$, $\sigma=\text{diag}([0.15,0.15,0.25,0.25,0.25,0.25,0.25,\\0.25,0.25,0.25,0.25,0.25,0.15,0.15,0.25,0.25,0.25,0.25,0.25,0.25,0.25,0.25,0.25,0.25,0.25]), R_x=\text{diag}([25,1,25,\\1,25,1,25,1,25,1,25,1,25,1,25,1,25,1,25,1,25,1,25,1,25])$, $R_u = 1$, $G=\text{diag}([25,25,25,25,25,25,1,25,1,25,1,25,25,25,25,25,25,25,1,25,\\1,25,1,25,1])$, $x_0=(0,1,0.1,\ldots,0.1)$. 
\begin{figure}[htp]
\centering
\begin{tabular}{ccc}
          \includegraphics[width=62mm]{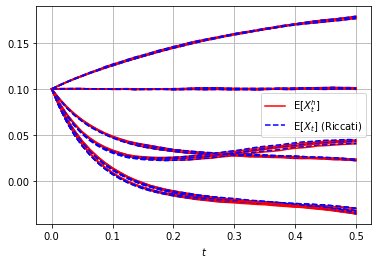}&  \includegraphics[width=62mm]{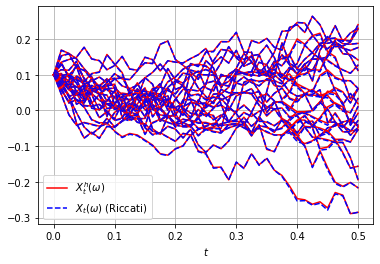}\\ \includegraphics[width=62mm]{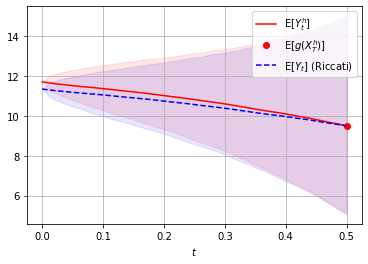}&
          \includegraphics[width=62mm]{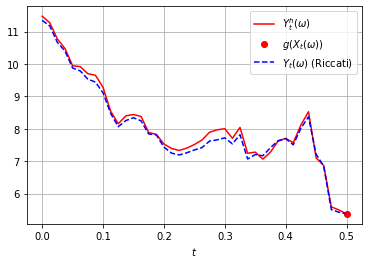}\\
          \includegraphics[width=62mm]{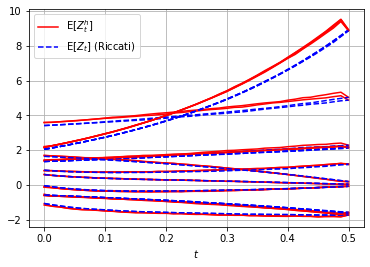}&   
         \includegraphics[width=62mm]{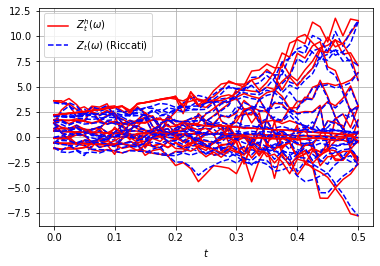}
\end{tabular}
\caption{Average of solutions and a single solution path compared to their analytic counterparts for the second LQ control problem from Section~\ref{d6ell2}, \textit{i.e.,} the problem with $d=25$ and $\ell=1$. The shaded area represents an empirical credible interval for $Y$, defined as the area between the 5:th and the 95:th percentiles at each time point. We do not include credible intervals for $X$ and $Z$ in this figure to facilitate visualization. For $X$ and $Z$, we see one realization of each of the 25 components.}\label{XYZd25ell1}
\end{figure}

\subsection{Non-linear quadratic control problems}\label{nonlinear}
Finally, we consider a control problem with non-linear coefficients in the state equation and quadratic coefficients in the cost functional. It has stable and unstable equilibrium points at the odd and even integers, respectively. The problem has been chosen to mimic the unstable problems that are commonly considered in control, such as inverted pendulums. 

Let $x_0\in\R^d$, $A,\Sigma,\in\R^{d\times d}$, $R_x,G\in\mathbbm{S}^d_+$,  $R_u\in\mathbbm{S}^\ell_+$ and $B\in\R^{d\times \ell}$ be of full rank and $C\in\R^d$. The state equation and cost functional of a non-linear quadratic control problem are given by
\begin{equation}\label{NLQR_md}
   \begin{dcases}  X_t=x_0+\int_0^t\big(A\sin(\pi CX_s) + Bu_s\big)\d s + \int_0^t\Sigma(\mathrm{1}_d+X_sX_s^\top)\d W_s, \\
    J^u(t,x)=\E^{t,x}\Big[\int_t^T (\langle R_x X_s, X_s\rangle + \langle R_u u_s,u_s\rangle)\d s+\langle G X_T, X_T\rangle\Big].
    \end{dcases}
\end{equation}
Due to the linear dependence of the control and the quadratic cost functional, the optimal feedback control is again given by
\begin{equation}\label{NLopt_control}
    u^*_t=-\frac{1}{2}R_u^{-1}B^T\text{D}_x V(t,X_t).
\end{equation}
Similar to above, $V$ is the solution to the associated HJB-equation. Again, by setting $Y_t=V(t,X_t)$ and $Z_t=\text{D}_xV(t,X_t)$, we obtain the FBSDE
\begin{equation}
    \begin{dcases}
     X_t = x_0+\int_0^t\big[A\sin(\pi CX_s)-\frac{1}{2}BR_u^{-1}R_u^{-1}B^\top Z_s\big]\d s + \int_0^t\Sigma(\mathrm{1}_d+X_sX_s^\top)\d W_s,\\
    Y_t =g(X_T) -\int_t^T\big(\langle R_x X_s,X_s\rangle-\frac{1}{4}\langle R_u^{-1}B^\top Z_s, B^\top Z_s\rangle\big)\d s + \int_0^t Z_s^\top\Sigma(\mathrm{1}_d+X_sX_s^\top) \d W_s.
    \end{dcases}\label{FBSDE_NLQG}
\end{equation}

Particularly, we consider a three-dimensional problem with control in two dimensions, \textit{i.e.,} $d=3$ and $\ell=2$ and use the following matrices for the state  \begin{align*}
A&=\text{diag}([1,1,1]),\quad B=\begin{pmatrix}1&0\\0&1\\1&1\end{pmatrix},\quad C=\text{diag}([1,1,1]),\\
    \Sigma&=\text{diag}([0.1,0.1,0.1]),\quad x_0=(0.1,0.1,0.1)^\top,\quad T=0.25.
\end{align*}
For the cost functional, we have the matrices
\begin{equation*}
        R_x=\text{diag}([5,1,1]),\quad R_u=\text{diag}([1,1]),\quad
    G=\text{diag}([1,5,1]).
\end{equation*}
Table \ref{table3} shows the  experimental order of convergence of the terminal condition and the initial value of the BSDE. The factor two between them is again consistent with Theorem~\ref{thm_Y02} and Corollary~\ref{thm_Y03}, even though the problem does not fall under the assumptions of these results.

\begin{table}[]
\center
\begin{tabular}{llllll}
\textbf{} &  \multicolumn{2}{l}{$\|Y_N^h-g(X^h_N)\|_{L^2(\Omega)}$} & \multicolumn{2}{l}{ $|Y_0-Y_0^h|$} & $Y_0^h$ \\ \hline
         $N$ &  Error & EOC & Error & EOC & Value  \\ \hline
         5 & 2.69e-2 & 0.59 &9.80e-3& 1.01 & 0.2297 \\
         10 &  1.79e-2 & 0.53 &4.85e-3&1.03& 0.2241\\
         20 &  1.24e-2 & 0.50 &2.38e-3&0.99& 0.2219\\
         40 & 8.76e-3 & 0.49 &1.20e-3&0.98& 0.2207\\
         80 &  6.27e-3 &   &6.07e-4& & 0.2200\\ \hline
\end{tabular}\label{table3}
\caption{Errors and experimental order of convergence for the nonlinear control problem in Section~\ref{nonlinear}. A reference solution of $Y_0=0.2194$ is computed with the same method on a fine grid with $N=160$ time points.}
\end{table}

\section*{Acknowledgments}
We very much thank Boualem Djehiche for his kindness in bringing some very useful references to our knowledge.
K.A. and C.W.O. acknowledge the funding of their research by the European Union, under the H2020-EU.1.3.1. MSCA-ITN-2018 scheme, Grant 813261.

\end{document}